%% file: paper.tex
\title{Explicit Multi-Matrix Topological Expansion for Quaternionic Random Matrices}
\author{C.\ E.\ I.\ Redelmeier\thanks{Research supported by a two-year Sophie Germain post-doctoral scholarship provided by the Fondation math\'{e}matique Jacques Hadarmard, held at the D\'{e}partement de Math\'{e}matiques, UMR 8628 Universit\'{e} Paris-Sud 11-CNRS, B\^{a}timent 425, Facult\'{e} des Sciences d'Orsay, Universit\'{e} Paris-Sud 11, F-91405 Orsay Cedex.}}

\documentclass[11pt]{article}

\usepackage{graphicx}
\usepackage{graphics}
\usepackage{amsfonts}
\usepackage{amssymb}
\usepackage{amsthm}
\usepackage{amsmath}
\usepackage{color}
\usepackage{cite}

\newtheorem{theorem}{Theorem}[section]
\newtheorem{lemma}[theorem]{Lemma}
\newtheorem{proposition}[theorem]{Proposition}
\newtheorem{corollary}[theorem]{Corollary}

\theoremstyle{remark}
\newtheorem{remark}[theorem]{Remark}
\newtheorem{notation}[theorem]{Notation}
\newtheorem{example}[theorem]{Example}

\theoremstyle{definition}
\newtheorem{definition}[theorem]{Definition}
\newtheorem{algorithm}[theorem]{Algorithm}

\begin{document}

\maketitle

\begin{abstract}
We present an explicit formula for the expected value of a product of several independent symplectically invariant matrices in which the trace and real part function may be applied, possibly to different subexpressions.  This takes the form of a topological expansion; however, each term has two topologies: one for the trace, and another for the real part.

The traces and real parts can always be written in terms of index contraction, but in some cases, it is possible to write the expression as a product in which the two functions are applied to bracketed intervals in a legal bracket diagram.  We present the conditions under which this may be done, and an algorithm to construct such an expression given the contracted indices when possible.

The summands in the topological expansion are written in terms of matrix cumulants.  We compute the matrix cumulants of quaternionic Ginibre, Gaussian symplectic, quaternionic Wishart, and Haar-distributed symplectic matrices, which allow direct computation of an expression constructed from several independent ensembles of any of these matrices.
\end{abstract}

\section{Introduction}

Quaternionic random matrices correspond to the \(\beta=4\) case, where \(\beta\) can be thought of as a sort of inverse temperature measuring the noise in the behaviour of the eigenvalues versus their sensitivity to predictable forces, such as their mutual repulsion.  (The real case is \(\beta=1\) and the complex \(\beta=2\).)  In quantum field theory, they correspond to fermionic particles (where the real case corresponds to bosonic particles).

It is often desirable to consider products of quaternion-valued traces, or other expressions in which the real part function (a function which resembles the trace, and which is in fact the normalized trace in the usual representation of quaternions as \(2\times 2\) matrices) is not applied to exactly the same subexpressions as the trace.  We consider a fairly large class of expressions: those which can be expressed by index contraction of the matrix indices, and, independently, the indices of the matrix representation of the quaternions (which corresponds to quaternion multiplication, or to taking the real part of the quaternion if the index contraction is cyclic).  (Our expression leaves at most two indices of each kind uncontracted, but could be generalized to a tensor expression.)  For example, we may calculate the expected value
\[\mathbb{E}\left[\mathrm{tr}\left(X_{3}X_{8}^{\ast}\mathrm{Re}\left(X_{4}\mathrm{Re}\left(X_{2}\right)X_{1}\right)\right)\mathrm{Re}\left(\mathrm{tr}\left(X_{5}X_{7}^{\ast}\mathrm{tr}\left(X_{6}\right)\right)\right)\right]\]
(where the \(X_{k}\) are random matrices whose joint distribution is symplectically invariant).  A real part of a matrix may be interpreted as a real-valued matrix, and the trace of a quaternion-valued matrix may be interpreted as a quaternion multiple of the identity matrix, either of which may be included in a matrix product expression.  The above expression is quaternion-valued.  It may be expressed in the Einstein summation convention:
\[2^{-3}N^{-3}\mathbb{E}\left[X^{\left(1\right)}_{ab;\alpha\beta}X^{\left(2\right)}_{ca;\gamma\gamma}X^{\left(3\right)}_{bd;\delta\epsilon}X^{\left(4\right)}_{ec;\beta\alpha}X^{\left(5\right)}_{fg;\zeta\eta}X^{\left(6\right)}_{hh;\theta\zeta}X^{\left(-7\right)}_{gf;\eta\theta}X^{\left(-8\right)}_{de;\epsilon\iota}\right]\textrm{.}\]
(We have moved the subscript to a bracketed superscript and used a negative sign to represent the conjugate transpose denoted above as \(X_{k}^{\ast}\), as we will continue to do.  We are also using Latin indices for the matrix indices and Greek for the indices of the matrix representation of quaternions.)  We note that there are two uncontracted (quaternion) indices, \(\delta\) and \(\iota\).  We may also calculate the expected value of an expression such as
\[\mathbb{E}\left[X^{\left(1\right)}_{ab;\alpha\beta}X^{\left(2\right)}_{cb;\gamma\delta}X^{\left(3\right)}_{cd;\beta\alpha}X^{\left(4\right)}_{da;\delta\gamma}\right]\textrm{,}\]
which cannot be expressed as a product of matrices where the functions \(\mathrm{Re}\) and \(\mathrm{tr}\) are applied to bracketed subexpressions.  One reason is that the orders of the matrix and quaternion indices of \(X_{2}\) prevent it from appearing as either \(X_{2}\) or \(X_{2}^{\ast}\).  It is also because a sort of crossing appears in the cycles of the trace and the real part.  However, the same techniques may be used to calculate its expected value.

We consider conditions on the two index contractions which permit a form like the first expression.  These conditions may be expressed in several ways, involving planarity conditions or geodesics in Cayley graphs.  We adapt an algorithm from \cite{MR0404045} to construct the expression from the index contractions when it is possible.

The expected value of such an expression resembles the topological expansions which are used to calculate complex and real matrix integrals.  As with the real case, the quaternionic case involves nonorientable ribbon graphs (see, e.g., \cite{MR2005857, MR2480549} in addition to the orientable graphs appearing in the complex case (see, e.g., \cite{MR2036721}, Chapter~3).

We consider expressions involving several independent quaternionic matrices.  This complicates the expression in ways unlike the complex and real cases, since the quaternionic trace is not cyclic (a property which depends on the commutativity of the field), so the related symmetries do not necessarily hold when the trace is applied to an expression which is not itself cyclically invariant.  This property has some surprising consequences in second-order freeness, which will be discussed in future work.

Section~\ref{section: preliminaries} contains the notation and lemmas we will need throughout.  In Section~\ref{section: index contraction} we present and prove the topological expansion formula, which gives an expansion for a fairly general class of quaternionic several-matrix expressions in terms of the matrix cumulants of the ensembles, allowing exact computation for expressions involving matrix models for which these have been computed.  In Section~\ref{section: bracket diagrams} we present a number of equivalent conditions that allow a term given in terms of contracted indices, such as those that appear in the above topological expansion, to be expressed as a bracket diagram in which the functions \(\mathrm{Re}\) and \(\mathrm{tr}\) are applied to bracketed subexpressions.  In Section~\ref{section: Gaussian} we compute monomial integration formulas and matrix cumulants of three matrix ensembles constructed from Gaussian matrices, specifically, quaternionic Ginibre matrices, Gaussian symplectic ensemble matrices, and quaternionic Wishart matrices.  In Section~\ref{section: Haar} we present the monomial integration formula and matrix cumulants of Haar-distributed symplectic matrices.  We use this to demonstrate that any symplectically invariant distribution satisfies the hypotheses of the topological expansion formula given in Section~\ref{section: index contraction}, and give the matrix cumulants in terms of the Weingarten function and the expected value of expressions where the trace and real part are taken together over products of the matrices.  We also consider the case where the matrices are symplectically in general position but not necessarily independent.

\section{Preliminaries}
\label{section: preliminaries}

\begin{notation}
We denote the set \(\left\{1,\ldots,n\right\}\) by \(\left[n\right]\).

For a set \(I\subseteq\mathbb{Z}\), we let \(-I=\left\{-k:k\in I\right\}\) and \(\pm I:=I\cup\left(-I\right)\).

We will often want to add a ``point at infinity'' to a set.  For a set \(I\), we let \(I_{\infty}:=I\cup\left\{\infty\right\}\).
\end{notation}

\begin{definition}
A {\em set partition} of a set \(I\) is a set of subsets \(V_{1},\ldots,V_{k}\subseteq I\) called {\em blocks} such that \(V_{i}\neq\emptyset\) for all \(i\in\left[k\right]\), \(V_{i}\cap V_{j}=\emptyset\) for all \(i\neq j\), and \(\bigcup_{i=1}^{k}V_{k}=I\).  We denote the set of partitions on set \(I\) by \({\cal P}\left(I\right)\) and the set of partitions on \(\left[n\right]\) by \({\cal P}\left(n\right)\).  We denote the number of blocks in a partition \(\pi\) by \(\#\left(\pi\right)\).

We define a poset on \({\cal P}\left(I\right)\) by letting \(\pi\preceq\rho\) if every block in \(\pi\) is a subset of a block in \(\rho\).  This poset is a lattice (see any standard reference on combinatorics, such as \cite{MR1311922}): any \(\pi,\rho\in{\cal P}\left(I\right)\) have a greatest lower bound \(\pi\wedge\rho\preceq\pi,\rho\) (called the {\em meet}) such that \(i,j\in I\) are in the same block of \(\pi\wedge\rho\) if and only if they are in the same block of \(\pi\) and the same block of \(\rho\), and a least upper bound \(\pi\vee\rho\succeq\pi,\rho\) (called the {\em join}) whose blocks are the union of the blocks of \(\pi\) which are connected by (intersect nontrivially with) blocks of \(\rho\).
\end{definition}

\begin{definition}
An {\em integer partition} \(\lambda\) of \(n\in\mathbb{Z}\) is a set of integers \(\lambda_{1},\ldots,\lambda_{k}\) with \(\lambda_{1}\geq\ldots\geq\lambda_{k}>0\) and \(\lambda_{1}+\cdots+\lambda_{k}=n\).

The {\em length} of an integer partition is the number of integers, in this case \(k\).
\end{definition}

It will usually be clear from context which type of partition we are referring to.  We will sometimes call a set partition a partition, but we will always specify an integer partition.

\subsection{Permutations}

\begin{notation}
We denote the set of all permutations on a set \(I\) by \(S\left(I\right)\) and the set of permutations on \(\left[n\right]\) by \(S_{n}\).  We use the convention that permutations act right-to-left.

We will typically use cycle notation, where we write the elements of each orbit of \(\pi\) in cyclic order \(k,\pi\left(k\right),\pi^{2}\left(k\right),\ldots,\pi^{-1}\left(k\right)\) in brackets.  We note that, if \(\pi,\rho\in S\left(I\right)\), then \(\rho\pi\rho^{-1}\) has the same cycle structure as \(\pi\), where each \(k\) is replaced by \(\rho\left(k\right)\).

The orbits of \(\pi\in S\left(I\right)\) form a partition of the set \(I\), which we will denote \(\Pi\left(\pi\right)\) (or, if it is clear from the context, we will use the permutation itself to represent this partition).  As with partitions, we will denote the number of orbits of \(\pi\) by \(\#\left(\pi\right)\).  We can define a distance metric by letting \(d\left(\pi,\rho\right):=n-\#\left(\rho\pi^{-1}\right)\).  We note that these quantities are unchanged by conjugation, and hence by cycling factors.
\end{notation}

The image of a permutation acting on a partition is the partition whose blocks are the images of the blocks of the original partition.

\begin{definition}
If \(\pi\in S\left(I\right)\) and \(J\subseteq I\), we define the {\em permutation induced by \(\pi\) on \(J\)}, denoted \(\left.\pi\right|_{J}\), by letting \(\left.\pi\right|_{J}\left(k\right)\) be the first of \(\pi\left(k\right),\pi^{2}\left(k\right),\ldots\) in \(J\).  We note that in cycle notation, \(\left.\pi\right|_{J}\) is \(\pi\) with the elements not in \(J\) deleted.
\end{definition}

\begin{definition}
We call a permutation {\em even} (resp.\ {\em odd}) if it can be written as the product of an even (resp.\ odd) number of transpositions.  We note that, since multiplication by a transposition changes the number of orbits by one, a permutation \(\pi\in S\left(I\right)\) is even (resp.\ odd) when \(\left|I\right|-\#\left(\pi\right)\) is even (resp.\ odd).  We write
\[\mathrm{sgn}\left(\pi\right):=\left(-1\right)^{\left|I\right|-\#\left(\pi\right)}\textrm{.}\]
\end{definition}
We note that the even permutations in form a subgroup of \(S\left(I\right)\) of index \(2\), and the odd permutations form the other coset.

\begin{definition}
We call a permutation \(\pi\) {\em alternating} if the sign of \(\pi\left(k\right)\) is always opposite that of \(k\).  We denote the set of alternating permutations on \(I\) by \(S_{\mathrm{alt}}\left(\pm\left[n\right]\right)\).
\end{definition}

\subsection{Parings, the Hyperoctahedral Group, and Maps on Unoriented Surfaces}

\begin{definition}
A {\em pairing} is a partition \(\pi\in{\cal P}\left(I\right)\) where every block contains exactly two elements; or equivalently, a permutation \(\pi\in S\left(I\right)\) in which each cycle has exactly two elements.
\end{definition}

The hyperoctahedral group is the group of symmetries of the \(n\)-dimensional equivalent of the octahedron (cross polytope), or equivalently of the \(n\)-hypercube:
\begin{definition}
For \(n\in\mathbb{N}\), we define a subgroup \(B_{n}\leq S_{2n}\) generated by the transpositions \(\left(2k-1,2k\right)\) and the permutations \(\left(2k_{1}-1,2k_{2}-1\right)\left(2k_{1},2k_{2}\right)\), \(k,k_{1},k_{2}\in\left[n\right]\) (i.e.\ the group of permutations which preserve the pairing \(\left\{\left\{1,2\right\},\ldots,\left\{2n-1,2n\right\}\right\}\)).  
\end{definition}
The two elements in the pairs may be thought of as the two ends of the \(n\) axes, which may be reversed or permuted amongst themselves.

We note that the cosets of \(B_{n}\) in \(S_{2n}\) correspond to \({\cal P}_{2}\left(n\right)\): we can find a permutation \(\pi\in S_{2n}\) which maps pairing \(\left\{\left\{1,2\right\},\ldots,\left\{2n-1,2n\right\}\right\}\) to any other pairing, and acting first by any element of \(B_{n}\) does not change the image.

We will use permutations to encode maps (see, e.g., \cite{MR0404045, MR1603700, MR2036721}), which are used in the computation of matrix integrals.  In the quaternionic case, these maps may not be orientable \cite{MR2005857, MR2480549}.  In order to represent these maps, we consider maps on the orientable covering space which are consistent with the covering map (see \cite{MR1867354}, pages 234--235 for the topological construction, and \cite{MR1813436, MR2851244, MR3217665, 2012arXiv1204.6211R} for the construction of the permutations).  The topological constructions are not necessary to our proofs, but motivate many of the operations.  Roughly, we use a cycle of a permutation to enumerate the edge-ends that appear (in counter-clockwise order) around a face, hyperedge (like an edge, but which may have one, two, or more ends, rather than two as an edge would have), or vertex, and a permutation to encode a collection of faces, hyperedges, or vertices.

\begin{notation}
We denote the function \(k\mapsto-k\) by \(\delta\).

We denote the set of permutations \(\pi\in S\left(\pm I\right)\) such that \(\delta\pi\delta=\pi^{-1}\) and such that \(k\) and \(-k\) do not appear in the same cycle by \(\mathrm{PM}\left(I\right)\).  We denote such permutations on \(\pm\left[n\right]\) by \(\mathrm{PM}\left(n\right)\).

The cycles of such a permutation \(\pi\) appear in pairs, where the order and sign of the integers are reversed.  For each pair, we may pick the cycle where the smallest absolute value integer (or infinity, if it appears) in the cycles appears as a positive integer.  We denote the product of these cycles by \(\mathrm{FD}\left(\pi\right)\), which we will consider a permutation on only the elements in those cycles.  We will also use this symbol to represent the set of elements appearing in those cycles.
\end{notation}
If \(\delta\pi\delta=\pi^{-1}\), to see that no \(k\) and \(-k\) appear in the same cycle, it is sufficient to check that \(\pi\left(k\right)\neq-k\) for all \(k\) \cite{MR3217665}.

Intuitively the positive and negative integers can be thought of as being on opposite sides of the same point.  Cycles of permutations represent cycles in counter-clockwise order, so viewed from the opposite side, the order and all signs are reversed.  See \cite{2012arXiv1204.6211R} for diagrams illustrating this intuition.

\begin{definition}
Given a permutation \(\varphi_{+}\in S\left(I\right)\) encoding faces (where \(k\) and \(-k\) are never both in \(I\)), and another \(\alpha\in\mathrm{PM}\left(I\right)\) encoding hyperedges, we can find another encoding the vertices.  Let \(\varphi_{-}=\delta\varphi_{+}\delta\).  Then
\[K\left(\varphi_{+},\alpha\right):=\varphi_{+}^{-1}\alpha^{-1}\varphi_{-}\textrm{.}\]

We define the Euler characteristic of \(\varphi_{+}\) and \(\alpha\) by
\[\chi\left(\varphi_{+},\alpha\right):=\#\left(\varphi_{+}\varphi_{-}^{-1}\right)/2+\#\left(\alpha\right)/2+\#\left(K\left(\varphi_{+},\alpha\right)\right)/2-\left|I\right|\textrm{.}\]
While the permutation \(K\left(\varphi_{+},\alpha\right)\) depends on the domain \(I\) of \(\varphi_{+}\), the cycle structure does not, so we may define functions which only depend on the number of cycles (such as the Euler characteristic) or the cycle structure (such as the Weingarten function, defined below) using \(\varphi^{-1}\alpha^{-1}\) instead of \(K\left(\varphi_{+},\alpha\right)\).  We will consider such functions defined for \(\varphi,\alpha\in\mathrm{PM}\left(I\right)\) even when we do not know the subset \(I\) of \(\pm I\) on which \(\varphi_{+}\) is defined.
\end{definition}

\begin{lemma}
Let \(\pi_{1},\pi_{2}\in{\cal P}_{2}\left(I\right)\). Then
\[\#\left(\pi_{1}\vee\pi_{2}\right)=\#\left(\mathrm{FD}\left(\pi_{2}\delta\pi_{1}\right)\right)=\#\left(\pi_{1}\pi_{2}\right)/2\textrm{.}\]
\label{lemma: pairings}
\end{lemma}
\begin{proof}
We may list the elements of a block of \(\pi_{1}\vee\pi_{2}\) by beginning with an element \(k\), and alternatingly applying \(\pi_{1}\) and \(\pi_{2}\).  (Since both are self-inverse, applying either twice gives an element that has already been listed.)

The permutation \(\pi_{2}\delta\pi_{1}\) is a premap: \(\pi_{2}\delta\pi_{1}\left(-k\right)=\pi_{2}\delta\pi_{1}\delta\left(k\right)=\delta\pi_{1}\delta\pi_{2}\left(k\right)=-\left(\pi_{2}\delta\pi_{1}\right)\left(k\right)\) (\(\pi_{2}\) and \(\delta\pi_{1}\delta\) commute, since the former acts nontrivially only on the positive integers and the latter only on the negative integers); and it cannot take \(k\) to \(-k\) since exactly one of \(\pi_{1}\) and \(\pi_{2}\) acts nontrivially (depending on the sign of \(k\), since neither has fixed points on the positive integers).  Since it is alternating, which of \(\pi_{1}\) and \(\pi_{2}\) acts nontrivially also alternates, so the absolute values of the elements of a cycle are the elements of a block \(\pi_{1}\vee\pi_{2}\).  The paired cycle also has this property, so two cycles of \(\pi_{2}\delta\pi_{1}\) correspond to a block of \(\pi_{1}\vee\pi_{2}\).

Applying the permutation \(\pi_{1}\pi_{2}\) gives every other element of a block as described above.  Since the blocks of \(\pi_{1}\vee\pi_{2}\) have an even number of elements (since they are a disjoint union of pairs), this will only give half of the elements, so each block of \(\pi_{1}\vee\pi_{2}\) is the disjoint union of two orbits of \(\pi_{1}\pi_{2}\).
\end{proof}

\begin{lemma}
The sets \({\cal P}_{2}\left(\pm\left[n\right]\right)\) and \(\mathrm{PM}\left(n\right)\) are in bijection under the map \(\pi\mapsto\delta\pi\).  Each pair in \(\pi\) contains exactly one element of \(\mathrm{FD}\left(\delta\pi\right)\).
\label{lemma: pairing bijection}
\end{lemma}
\begin{proof}
Let \(\pi\in{\cal P}_{2}\left(\pm\left[n\right]\right)\).  Then \(\left(\delta\pi\right)^{-1}=\delta\delta\pi\delta\).  Furthermore, \(\pi\left(k\right)\neq k\), so \(\delta\pi\left(k\right)\neq -k\).  Conversely, if \(\rho\in\mathrm{PM}\left(n\right)\), then \(\left(\delta\rho\right)^{2}=\rho^{-1}\rho\) and it has no fixed points since \(\rho\left(k\right)\neq\delta\left(k\right)\).

If \(k\in\mathrm{FD}\left(\delta\pi\right)\), then so is \(\delta\pi\left(k\right)\), so \(\pi\left(k\right)\) is not.
\end{proof}
In particular, this gives us a way to index the pairs in a pairing: we index over elements in \(\mathrm{FD}\left(\delta\pi\right)\).  Lemma~\ref{lemma: parity} below may be used to index the pairs of two pairings in \({\cal P}\left(n\right)\) simultaneously, by negative elements of \(\mathrm{FD}\left(\pi_{2}\delta\pi_{1}\right)\).

The following technical lemma will be useful in several constructions depending on both the coset and the sign of a permutation:

\begin{lemma}
Let \(\pi_{1},\pi_{2}\in{\cal P}_{2}\left(n\right)\).  Then we can find \(\sigma_{1},\sigma_{2}\in S_{n}\) such that \(\sigma_{i}\left(\left\{\left\{1,2\right\},\ldots\left\{n-1,n\right\}\right\}\right)=\pi_{i}\) and \(\sigma_{i}\left(\left\{1,3,\ldots,n-1\right\}\right)=\left\{k>0:-k\in\mathrm{FD}\left(\pi_{2}\delta\pi_{1}\right)\right\}\), \(i=1,2\).

Furthermore, \(\mathrm{sgn}\left(\sigma_{2}\sigma_{1}^{-1}\right)=\left(-1\right)^{n/2-\#\left(\pi_{1}\vee\pi_{2}\right)}\).
\label{lemma: parity}
\end{lemma}
\begin{proof}
If \(-k\in\mathrm{FD}\left(\pi_{2}\delta\pi_{1}\right)\) with \(k>0\), then \(\pi_{1}\left(k\right)=\left(\pi_{2}\delta\pi_{1}\right)^{-1}\left(-k\right)\) and \(\pi_{2}\left(k\right)=\pi_{2}\delta\pi_{1}\left(-k\right)\), in each case positive numbers since \(\pi_{2}\delta\pi_{1}\) is alternating.  Thus the \(\pi_{i}\) pair each \(k\in\left[n\right]\) where \(-k\in\mathrm{FD}\left(\pi_{2}\delta\pi_{1}\right)\) with an \(l\in\left[n\right]\) such that \(l\in\mathrm{FD}\left(\pi_{2}\delta\pi_{1}\right)\).  Thus we may construct \(\sigma_{i}\) to take each odd number to a \(k\) with \(-k\in\mathrm{FD}\left(\pi_{2}\delta\pi_{1}\right)\) and its even successor to the partner of \(k\) in \(\pi_{i}\).

A permutation in the same coset of opposite sign will be \(\sigma_{i}h\) for some \(h\in B_{n/2}\) with at least one factor of the form \(\left(2k-1, 2k\right)\), so it maps at least one even number to the image of an odd number under \(\sigma_{i}\), and thus does not satisfy the given constraints.  Thus the constraints determine the signs of the \(\sigma_{i}\).  In the above construction, for \(k\in\left[n\right]\) with \(-k\in\mathrm{FD}\left(\pi_{2}\delta\pi_{1}\right)\), we have \(\sigma_{2}\sigma_{1}^{-1}\left(k\right)=k\) (giving \(n/2\) cycles consisting of a single element).  For \(l\in\left[n\right]\) with \(l\in\mathrm{FD}\left(\pi_{2}\delta\pi_{1}\right)\), \(\sigma_{2}\sigma_{1}^{-1}\left(l\right)=\pi_{2}\pi_{1}\left(l\right)\) (accounting for \(\#\left(\pi_{1}\vee\pi_{2}\right)\) cycles, from Lemma~\ref{lemma: pairings}).  The second part of the lemma follows.
\end{proof}

The following corollary follows:
\begin{corollary}
Let \(I\) be a finite set, and let \(\pi_{1},\pi_{2}\in{\cal P}_{2}\left(I\right)\) where each pair in each \(\pi_{i}\) has a distinguished element.  Let \(\rho\in S_{I}\) such that \(\rho\left(\pi_{1}\right)=\pi_{2}\).  Let \(m\) be the number of \(x\in I\) where \(x\) is the distinguished member of its pair in both \(\pi_{1}\) and \(\pi_{2}\).  Then
\[\mathrm{sgn}\left(\rho\right)=\left(-1\right)^{\#\left(\pi_{1}\vee\pi_{2}\right)+m}\textrm{.}\]
\label{corollary: parity}
\end{corollary}

\subsection{Quaternions}

\begin{definition}
The set of {\em quaternions}, denoted \(\mathbb{H}\), are the linear combinations of elements \(1\), \(i\), \(j\), and \(k\), defined by the relations \(i^{2}=j^{2}=k^{2}=-1\) and \(ij=k\), \(jk=i\), and \(ki=j\).

The real part of a quaternion \(a+bi+cj+dk\) (\(a,b,c,d\in\mathbb{R}\)) is:
\[\mathrm{Re}\left(a+bi+cj+dk\right)=a\textrm{.}\]
The conjugate is:
\[\overline{a+bi+cj+dk}=a-bi-cj-dk\textrm{.}\]
\end{definition}

We note that the quaternion \(a+bi+cj+dk\) may be represented as the \(2\times 2\) matrix
\[\left(\begin{array}{cc}a+bi&c+di\\-c+di&a-bi\end{array}\right)\textrm{;}\]
and that the real part is \(\frac{1}{2}\) times the trace of the matrix (a normalized trace; see below), and the matrix representation of the conjugate is the adjoint matrix.

We will index the columns and rows with \(1\) and \(-1\).  We will often make use of the following identity:
\begin{equation}
\left[\overline{Q}\right]_{\eta\theta}=\eta\theta\left[Q\right]_{-\theta,-\eta}\textrm{.}
\label{formula: annoying sign}
\end{equation}

\subsection{Matrices}

The \(N\times N\) quaternionic matrices may be represented as \(M_{N\times N}\left(\mathbb{C}\right)\otimes_{\mathbb{C}}\mathbb{H}\).  We may pick out an entry with two sets of indices: two matrix indices and two indices of the \(2\times 2\) matrix representing the quaternion.  We will separate the two sets with a semicolon.

If \(A\in M_{N\times N}\left(\mathbb{H}\right)\), we denote the normalized trace
\[\mathrm{tr}\left(A\right):=\frac{1}{N}\mathrm{Tr}\left(A\right)=\frac{1}{N}\left(A_{11}+\cdots+A_{NN}\right)\textrm{.}\]
Since the quaternions are not commutative, the trace does not have the property that \(\mathrm{Tr}\left(AB\right)=\mathrm{Tr}\left(BA\right)\).

If \(A\in M_{N\times N}\left(\mathbb{H}\right)\), then the adjoint \(A^{\ast}\) is the conjugate transpose: we take the transpose and the entrywise conjugate.

We will often use subscripts on the indices, such as \(\iota_{k}\), \(k\in\pm\left[n\right]\).  We will consider this a function \(\iota:\pm\left[n\right]\rightarrow\left[N\right]\).  This allows us to use function notation, such as the circle \(\circ\) for composition.  We will also use a dot \(\cdot\) to indicate the pointwise multiplication of functions.

With commutative fields, we have the following folklore lemma (which may be proven by computation):
\begin{lemma}
Let \(A_{1},\ldots,A_{n}\in M_{N\times N}\left(\mathbb{F}\right)\) (where \(\mathbb{F}\) is a field) and \(\pi\in S_{n}\).  Then
\[\mathrm{Tr}_{\pi}\left(A_{1},\ldots,A_{n}\right)=\sum_{i:\left[n\right]\rightarrow\left[N\right]}\prod_{k=1}^{n}A^{\left(k\right)}_{i_{k},i_{\pi\left(k\right)}}\textrm{.}\]
\label{lemma: traces}
\end{lemma}
If we instead take \(\pi\in\mathrm{PM}\left(n\right)\) and interpret negative subscripts as transposes (for now, since we are not yet considering the second pair of indices), we can rewrite the matrix entry as \(A^{\left(k\right)}_{i_{\delta\mathrm{FD}\left(\pi\right)\delta\left(k\right)}i_{\mathrm{FD}\left(\pi\right)\left(k\right)}}\).  We note that \(A^{\left(-k\right)}_{i_{\delta\mathrm{FD}\left(\pi\right)\delta\left(-k\right)}i_{\mathrm{FD}\left(\pi\right)\left(-k\right)}}=A^{\left(k\right)}_{i_{\delta\mathrm{FD}\left(\pi\right)\delta\left(k\right)}i_{\mathrm{FD}\left(\pi\right)\left(k\right)}}\), so we can take the product over any \(I\subseteq\pm\left[n\right]\), rather than just \(\left[n\right]\).  Renumbering the left index \(i_{k}\), the subscript of the right index will be \(\delta\mathrm{FD}\left(\pi\right)^{-1}\delta\mathrm{FD}\left(\pi\right)\left(k\right)=\pi\left(k\right)\).  Then for \(\pi\in\mathrm{PM}\left(n\right)\) and \(I\subseteq\pm\left[n\right]\),
\[\mathrm{Tr}_{\mathrm{FD}\left(\pi\right)}\left(A_{1},\ldots,A_{n}\right)=\sum_{i:I\rightarrow\left[N\right]}\prod_{k\in I}A^{\left(k\right)}_{i_{k}i_{\pi\left(k\right)}}\textrm{.}\]

It is no longer as natural to take traces over permutations, since the quaternionic trace is not cyclic.  It turns out, however, that it is possible to take \(\mathrm{tr}\) and \(\mathrm{Re}\) over two possibly different permutations.  Introducing the second pair of indices, it is more convenient to take the product over the set \(\mathrm{FD}\left(\pi\right)\) where the quaternionic product and \(\mathrm{Re}\) are taken over \(\pi\), since conjugation affects the sign of the entry while transposition only reverses the order of the matrix indices.  We are led to the following definition:
\begin{definition}
Let \(\pi,\rho\in\mathrm{PM}\left(n\right)\).  Then
\begin{multline*}
\mathrm{Re}_{\mathrm{FD}\left(\pi\right)}\mathrm{tr}_{\mathrm{FD}\left(\rho\right)}\left(A_{1},\ldots,A_{n}\right)\\:=2^{-\#\left(\pi\right)}N^{-\#\left(\rho\right)}\sum_{\substack{i:\mathrm{FD}\left(\pi\right)\rightarrow\left[N\right]\\h:\mathrm{FD}\left(\pi\right)\rightarrow\left\{1,-1\right\}}}\prod_{k\in\mathrm{FD}\left(\pi\right)}A^{\left(k\right)}_{i_{k},i_{\rho\left(k\right)};h_{k},h_{\pi\left(k\right)}}\textrm{.}
\end{multline*}
(We note that there is only one \(\rho\) corresponding to a given \(\mathrm{FD}\left(\rho\right)\), so it is possible to consider \(\rho\left(k\right)\) for \(k\) not in \(\mathrm{FD}\left(\rho\right)\), as we do in the \(i\) indices.)

It will sometimes be convenient to consider a quaternion-valued trace (or similarly a matrix-valued expression).  We introduce a ``point at infinity'' by considering \(\pi,\rho\in\mathrm{PM}\left(\left[n\right]_{\infty}\right)\).  We consider the indices with subscript \(\infty\) to be uncontracted.  (If all indices are contracted, the point at infinity is in a cycle by itself.)  If \(\pi,\rho\in\left[n\right]_{\infty}\), then
\begin{multline*}
\left[\mathrm{Re}_{\pi}\mathrm{tr}_{\rho}\left(A_{1},\ldots,A_{n}\right)\right]_{i_{\infty},i_{\rho\left(\infty\right)};h_{\infty},h_{\pi\left(\infty\right)}}\\:=2^{-\left(\#\left(\pi\right)-1\right)}N^{-\left(\#\left(\rho\right)-1\right)}\sum_{\substack{i:\mathrm{FD}\left(\pi\right)\setminus\left\{\infty,\pm\rho\left(\infty\right)\right\}\rightarrow\left[N\right]\\h:\mathrm{FD}\left(\pi\right)\setminus\left\{\infty,\pi\left(\infty\right)\right\}\rightarrow\left\{1,-1\right\}}}\prod_{k\in\mathrm{FD}\left(\pi\right)\setminus\left\{\infty\right\}}A^{\left(k\right)}_{i_{k},i_{\rho\left(k\right)};h_{k},h_{\pi\left(k\right)}}\textrm{.}
\end{multline*}
\label{definition: traces}
\end{definition}

The expression considered in the introduction
\[\mathbb{E}\left[\mathrm{tr}\left(X_{3}X_{8}^{\ast}\mathrm{Re}\left(X_{4}\mathrm{Re}\left(X_{2}\right)X_{1}\right)\right)\mathrm{Re}\left(\mathrm{tr}\left(X_{5}X_{7}^{\ast}\mathrm{tr}\left(X_{6}\right)\right)\right)\right]\]
is then expressed as
\[\mathrm{Re}_{\varphi_{\mathrm{Re}}}\mathrm{tr}_{\varphi_{\mathrm{tr}}}\left(X_{1},X_{2},X_{3},X_{4},X_{5},X_{6},X_{7},X_{8}\right)\]
where \(\varphi_{\mathrm{Re}}=\left(\infty,3,-8,1\right)\left(4,1\right)\left(2\right)\left(5,-7,6\right)\) and \(\varphi_{\mathrm{tr}}=\left(\infty\right)\left(3,-8,4,2,1\right)\left(5,-7\right)\left(6\right)\).  The other expression considered in the introduction 
\[\mathbb{E}\left[X^{\left(1\right)}_{ab;\alpha\beta}X^{\left(2\right)}_{cb;\gamma\delta}X^{\left(3\right)}_{cd;\beta\alpha}X^{\left(4\right)}_{da;\delta\gamma}\right]\textrm{,}\]
is expressed
\[2^{2}N\mathrm{Re}_{\varphi_{\mathrm{Re}}}\mathrm{tr}_{\varphi_{\mathrm{tr}}}\left(X_{1},X_{2},X_{3},X_{4}\right)\]
where \(\varphi_{\mathrm{Re}}=\left(\infty\right)\left(1,3\right)\left(2,4\right)\) and \(\varphi_{\mathrm{tr}}=\left(\infty\right)\left(1,-2,3,4\right)\).  We note also that the domains of the two permutations are not the same here.

Our notation allows for at most two uncontracted indices of each type; however, by adjoining several infinities this notation and the results could be generalized to more complicated tensor expressions.

\section{Index Contraction Formulation}
\label{section: index contraction}

We present here our formula for the expected value of a quaternionic matrix expression.  Any symplectically invariant distribution satisfies the hypotheses of Proposition~\ref{proposition: topological expansion} (see Section~\ref{subsection: symplectically invariant}).  In particular, it is satsified by the identity matrix, so it is possible to express a more general class of expressions by letting some of the \(X_{k}\) be the identity matrix.  For the identity matrix, we calculate that 
\[f\left(\alpha\right)=\left\{\begin{array}{ll}1\textrm{,}&\alpha=e\\0\textrm{,}&\textrm{otherwise}\end{array}\right.\textrm{.}\]

\begin{proposition}
Let each \(c\in\left[C\right]\) be associated with a set of \(N\times N\) matrices \(\left\{X_{c,1},\ldots,X_{c,n}\right\}\) (independent from each other set) such that, for any \(I\subseteq\left[n\right]\),
\begin{multline*}
\mathbb{E}\left(\prod_{k\in I}X^{\left(c,k\right)}_{\iota_{k},\iota_{-k};\eta_{k},\eta_{-k}}\right)
\\=\sum_{\substack{\alpha\in\mathrm{PM}\left(I\right)\\\iota=\iota\circ\delta\alpha\\\mathrm{sgn}\cdot\eta=\mathrm{sgn}\circ\alpha\cdot\eta\circ\delta\alpha}}\left[\prod_{\substack{k\in\left[n\right]\\-k\in\mathrm{FD}\left(\alpha\right)}}\eta_{k}\eta_{-k}\right]\left(2N\right)^{\#\left(\alpha\right)/2-n}f_{c}\left(\alpha\right)
\label{equation: expectation of polynomials}
\end{multline*}
for some function \(f_{c}:\mathrm{PM}\left(I\right)\rightarrow\mathbb{C}\).

Let \(\varphi_{\mathrm{Re}},\varphi_{\mathrm{tr}}\in S_{n}\).  Let \(w:\left[n\right]\rightarrow\left[C\right]\), and let \(X_{k}=X_{w\left(k\right),k}\).  Let \(\varepsilon:\left[n\right]\rightarrow\left\{1,-1\right\}\).  Let \(Y_{1},\ldots,Y_{n}\) be quaternionic random matrices independent from the \(X_{k}\).  Then
\begin{multline*}
\mathbb{E}\left(\mathrm{Re}_{\varphi_{\mathrm{Re}}}\mathrm{tr}_{\varphi_{\mathrm{tr}}}\left(X_{1}^{\left(\varepsilon_{1}\right)}Y_{1},\ldots,X_{n}^{\left(\varepsilon_{n}\right)}Y_{n}\right)\right)\\=\sum_{\substack{\alpha=\alpha_{1}\cdots\alpha_{C}\\\alpha_{c}\in\mathrm{PM}\left(w^{-1}\left(c\right)\right)}}\left(-2\right)^{\chi\left(\varphi_{\mathrm{Re}},\delta_{\varepsilon}\alpha\delta_{\varepsilon}\right)-2\#\left(\varphi_{\mathrm{Re}}\right)}N^{\chi\left(\varphi_{\mathrm{tr}},\delta_{\varepsilon}\alpha\delta_{\varepsilon}\right)-2\#\left(\varphi_{\mathrm{tr}}\right)}\\f_{1}\left(\alpha_{1}\right)\cdots f_{C}\left(\alpha_{C}\right)\mathbb{E}\left(\mathrm{Re}_{\mathrm{FD}\left(K\left(\varphi_{\mathrm{Re}},\delta_{\varepsilon}\alpha\delta_{\varepsilon}\right)^{-1}\right)}\mathrm{tr}_{\mathrm{FD}\left(K\left(\varphi_{\mathrm{tr}},\delta_{\varepsilon}\alpha\delta_{\varepsilon}\right)^{-1}\right)}\left(Y_{1},\ldots,Y_{n}\right)\right)\textrm{.}
\end{multline*}
\label{proposition: topological expansion}
\end{proposition}
\begin{proof}

We divide the proof into three sections.  In the first we establish that the constraints on the indices are those expected according to the statement of the theorem.  In the second we show that the signs of the terms are consistent.  In the third we calculate the sign of each term.

\paragraph{Contracted indices}

We expand the left-hand side according to the definition.  We use (\ref{formula: annoying sign}) to choose our indices so that each random matrix \(X_{k}\) appears with indices \(\iota_{k},\iota_{-k};\eta_{k},\eta_{-k}\); i.e., if it is \(X_{k}\) which appears we take these indices, and if it is \(X_{k}^{\ast}\) we take
\begin{equation}
X^{\left(-k\right)}_{\iota_{-k},\iota_{k};\eta_{-k},\eta_{k}}=\eta_{k}\eta_{-k}X^{\left(k\right)}_{\iota_{k},\iota_{-k};-\eta_{k},-\eta_{-k}}\textrm{.}
\label{formula: faces sign}
\end{equation}
Which \(\iota\) indices appear on a given \(Y_{k}\) depends on whether \(X_{k}\) and \(X_{\varphi_{\mathrm{tr}}\left(k\right)}\) appear with or without a star (i.e., on the value of \(\varepsilon\) on \(k\) and \(\varphi_{\mathrm{tr}}\left(k\right)\)); which \(\eta\) indices, as well as whether their signs are reversed, depend on whether \(X_{k}\) and \(X_{\varphi_{\mathrm{Re}}\left(k\right)}\) appear with stars.  We find that \(Y_{k}\) appears with indices
\[\iota_{-\delta_{\varepsilon}\left(k\right)},\iota_{\delta_{\varepsilon}\varphi_{\mathrm{tr}}\left(k\right)};\varepsilon\left(k\right)\eta_{-\delta_{\varepsilon}\left(k\right)},\varepsilon\left(\varphi_{\mathrm{Re}}\left(k\right)\right)\eta_{\delta_{\varepsilon}\varphi_{\mathrm{Re}}\left(k\right)}\textrm{.}\]
We thus have
\begin{multline}
\mathbb{E}\left(\mathrm{Re}_{\varphi_{\mathrm{Re}}}\mathrm{tr}_{\varphi_{\mathrm{tr}}}\left(X_{1}^{\left(\varepsilon_{1}\right)}Y_{1},\ldots,X_{n}^{\left(\varepsilon_{n}\right)}Y_{n}\right)\right)
\\=2^{-\#\left(\varphi_{\mathrm{Re}}\right)}N^{-\#\left(\varphi_{\mathrm{tr}}\right)}\sum_{\substack{\iota:\left[n\right]\rightarrow\left[N\right]\\\eta:\left[n\right]\rightarrow\left\{1,-1\right\}}}\left[\prod_{k\in\left[n\right]:\varepsilon\left(k\right)=-1}\eta_{k}\eta_{-k}\right]\mathbb{E}\left[\prod_{k=1}^{n}X^{\left(k\right)}_{\iota_{k},\iota_{-k};\eta_{k},\eta_{-k}}\right]\\\times\mathbb{E}\left[\prod_{k=1}^{n}Y^{\left(k\right)}_{\iota_{-\delta_{\varepsilon}\left(k\right)},\iota_{\delta_{\varepsilon}\varphi_{\mathrm{tr}}\left(k\right)};\varepsilon\left(k\right)\eta_{-\delta_{\varepsilon}\left(k\right)},\varepsilon\left(\varphi_{\mathrm{Re}}\left(k\right)\right)\eta_{\delta_{\varepsilon}\varphi_{\mathrm{Re}}\left(k\right)}}\right]\textrm{.}
\label{formula: index expansion}
\end{multline}

According to the hypotheses, for a given \(\iota:\left[n\right]\rightarrow\left[N\right]\) and \(\eta:\left[n\right]\rightarrow\left\{1,-1\right\}\), the first expected value in (\ref{formula: index expansion}) is
\begin{align}
&\mathbb{E}\left[\prod_{k=1}^{n}X^{\left(k\right)}_{\iota_{k},\iota_{-k};\eta_{k},\eta_{-k}}\right]
\nonumber
\\&=\left(2N\right)^{\#\left(\alpha\right)/2-n}\prod_{c\in\left[C\right]}\sum_{\substack{\alpha_{c}\in\mathrm{PM}\left(w^{-1}\left(c\right)\right)\\\iota=\iota\circ\delta\alpha_{c}\\\mathrm{sgn}\cdot\eta=\mathrm{sgn}\circ\alpha\cdot\eta\circ\delta\alpha_{c}}}\left[\prod_{\substack{k\in\left[n\right]\\-k\in\mathrm{FD}\left(\alpha_{c}\right)}}\eta_{k}\eta_{-k}\right]f_{c}\left(\alpha_{c}\right)
\nonumber
\\&=\left(2N\right)^{\#\left(\alpha\right)/2-n}\sum_{\substack{\alpha=\alpha_{1}\cdots\alpha_{C}\\\alpha_{c}\in\mathrm{PM}\left(w^{-1}\left(c\right)\right)\\\iota=\iota\circ\delta\alpha\\\mathrm{sgn}\cdot\eta=\mathrm{sgn}\circ\alpha\cdot\eta\circ\delta\alpha}}\left[\prod_{\substack{k\in\left[n\right]\\-k\in\mathrm{FD}\left(\alpha\right)}}\eta_{k}\eta_{-k}\right]f_{1}\left(\alpha_{1}\right)\cdots f_{C}\left(\alpha_{C}\right)\textrm{.}
\label{formula: edges sign}
\end{align}
If we substitute this value into (\ref{formula: index expansion}), the summation conditions become a further set of constraints on the values of \(\iota\) and \(\eta\).

We show first that the indices appearing on the \(Y_{k}\) are as expected according to Definition~\ref{definition: traces}.  We take \(\alpha\) to be fixed and define \(\sigma_{\mathrm{tr}}:=K\left(\varphi_{\mathrm{tr}},\delta_{\varepsilon}\alpha\delta_{\varepsilon}\right)\) and \(\sigma_{\mathrm{Re}}:=K\left(\varphi_{\mathrm{Re}},\delta_{\varepsilon}\alpha\delta_{\varepsilon}\right)\).

The value of the entry of \(Y_{k}\) which appears in the remaining expected value expression may be obtained from \(Y_{-k}=Y^{\ast}\) using (\ref{formula: annoying sign}).  The orders of both the \(\iota\) and the \(\eta\) indices will be reversed, the signs of the \(\eta\) indices will be reversed, and we gain a factor which is the product of the \(\eta\) indices.  We let \(\varphi_{\mathrm{Re}+}:=\varphi_{\mathrm{Re}}\) and \(\varphi_{\mathrm{Re}-}:=\delta\varphi_{\mathrm{Re}}\delta\), and likewise with \(\varphi_{\mathrm{tr}}\).  We may easily verify (by checking two cases: \(k>0\) and \(k<0\)) that for any \(k\in\pm\left[n\right]\) the following indices of \(Y_{k}\) give the entry of \(Y_{\left|k\right|}\) appearing in (\ref{formula: index expansion}) (up to a factor from the \(\eta\) indices):
\[\iota_{\delta\delta_{\varepsilon}\varphi_{\mathrm{tr}-}\left(k\right)},\iota_{\delta_{\varepsilon}\varphi_{\mathrm{tr}+}\left(k\right)};\mathrm{sgn}\left(\delta_{\varepsilon}\varphi_{\mathrm{Re}-}\left(k\right)\right)\eta_{\delta\delta_{\varepsilon}\varphi_{\mathrm{Re}-}\left(k\right)},\mathrm{sgn}\left(\delta_{\varepsilon}\varphi_{\mathrm{Re}+}\left(k\right)\right)\eta_{\delta_{\varepsilon}\varphi_{\mathrm{Re}+}\left(k\right)}\textrm{.}\]
We may now write the product of \(Y_{k}\) entries in the second expected value on the right-hand side of (\ref{formula: index expansion}) as a product over \(\mathrm{FD}\left(\sigma_{\mathrm{Re}}\right)\), gaining a factor of 
\begin{equation}
\prod_{\substack{k\in\left[n\right]\\-k\in \mathrm{FD}\left(\sigma_{\mathrm{Re}}\right)}}\mathrm{sgn}\left(\delta_{\varepsilon}\varphi_{\mathrm{Re}-}\left(k\right)\right)\eta_{\delta\delta_{\varepsilon}\varphi_{\mathrm{Re}-}\left(k\right)},\mathrm{sgn}\left(\delta_{\varepsilon}\varphi_{\mathrm{Re}+}\left(k\right)\right)\eta_{\delta_{\varepsilon}\varphi_{\mathrm{Re}+}\left(k\right)}\textrm{.}
\label{formula: vertices sign}
\end{equation}
Since \(\iota=\iota\circ\delta\alpha\), the right \(\iota\) index of \(Y_{k}\) is equal to \(\iota_{\delta\alpha\delta_{\varepsilon}\varphi_{\mathrm{tr}+}\left(k\right)}=\iota_{\delta\delta_{\varepsilon}\varphi_{\mathrm{tr}-}\sigma_{\mathrm{tr}}^{-1}\left(k\right)}\), i.e.\ the left \(\iota\) index of \(Y_{\sigma_{\mathrm{tr}}^{-1}\left(k\right)}\).  Likewise, since \(\mathrm{sgn}\cdot\eta=\mathrm{sgn}\circ\alpha\cdot\eta\circ\delta\alpha\), the right \(\eta\) index of \(Y_{k}\) is equal to \(\mathrm{sgn}\left(\alpha\delta_{\varepsilon}\varphi_{\mathrm{Re}+}\left(k\right)\right)\eta_{\delta\alpha\delta_{\varepsilon}\varphi_{\mathrm{Re}+}\left(k\right)}\), which is the left \(\eta\) index of \(Y_{\sigma_{\mathrm{Re}}\left(k\right)}\).  Thus the indices are as we would expect, and the powers of \(2\) and \(N\) follow.

\paragraph{Consistency of the sign}

The remainder of the proof concerns the sign contributed by the \(\eta\) index multipliers appearing in (\ref{formula: faces sign}), (\ref{formula: edges sign}), and (\ref{formula: vertices sign}).  In this section we show that the sign of every term associated with a given \(\alpha\in\mathrm{PM}\left(n\right)\) is the same, so the terms may indeed be summed to a multiple of \(\mathrm{Re}_{\sigma_{\mathrm{Re}}}\mathrm{tr}_{\sigma_{\mathrm{tr}}}\left(Y_{1},\ldots,Y_{n}\right)\).

For a pair \(\left(k,\delta\alpha\left(k\right)\right)\) in pairing \(\delta\alpha\) with \(k\in\mathrm{FD}\left(\alpha\right)\), let \(\eta_{0}:=\mathrm{sgn}\left(k\right)\eta_{k}=\mathrm{sgn}\left(\alpha\left(k\right)\right)\eta_{\delta\alpha\left(k\right)}\).  Writing \(\eta_{k}\) and \(\eta_{\delta\alpha\left(k\right)}\) in terms of \(\eta_{0}\), \(\eta_{0}\) will appear in (\ref{formula: faces sign}) an odd number of times (i.e.\ once, as opposed to zero times or twice) exactly when \(\varepsilon\left(k\right)\neq\varepsilon\left(\alpha\left(k\right)\right)\).  It will appear an odd number of times in (\ref{formula: edges sign}) exactly when an odd number of \(-\left|k\right|\) and \(-\left|\alpha\left(k\right)\right|\) appear in \(\mathrm{FD}\left(\alpha\right)\), i.e.\ \(\mathrm{sgn}\left(k\right)\neq\mathrm{sgn}\left(\alpha\left(k\right)\right)\).  Thus it will appear an odd number of times in (\ref{formula: faces sign}) and (\ref{formula: edges sign}) when \(\mathrm{sgn}\left(\delta_{\varepsilon}\left(k\right)\right)\neq\mathrm{sgn}\left(\delta_{\varepsilon}\alpha\left(k\right)\right)\).

In (\ref{formula: vertices sign}), the indices containing \(\eta_{0}\) are the right and left \(\eta\) indices (respectively) of the matrices subscripted either \(\varphi_{\mathrm{Re}+}^{-1}\delta_{\varepsilon}\left(k\right)\) and \(\varphi_{\mathrm{Re}-}^{-1}\delta_{\varepsilon}\alpha\left(k\right)=\sigma_{\mathrm{Re}}^{-1}\varphi_{\mathrm{Re}+}^{-1}\delta_{\varepsilon}\left(k\right)\) or \(\delta\varphi_{\mathrm{Re}-}^{-1}\delta_{\varepsilon}\alpha\left(k\right)\) and \(\delta\varphi_{\mathrm{Re}+}^{-1}\delta_{\varepsilon}\left(k\right)=\sigma_{\mathrm{Re}}^{-1}\delta\varphi_{\mathrm{Re}-}^{-1}\delta_{\varepsilon}\alpha\left(k\right)\).  (The paired integers are in the same orbit of \(\sigma_{\mathrm{Re}}\) and the pairs are negatives of each other, so exactly one pair is contained in \(\mathrm{FD}\left(\sigma_{\mathrm{Re}}\right)\).)  This pair contributes an odd number of \(\eta_{0}\) to (\ref{formula: vertices sign}) exactly when \(\mathrm{sgn}\left(\varphi_{\mathrm{Re}+}^{-1}\delta_{\varepsilon}\left(k\right)\right)\neq\mathrm{sgn}\left(\varphi_{\mathrm{Re}-}^{-1}\delta_{\varepsilon}\alpha\left(k\right)\right)\) (or, equivalently, when their negatives are of opposite sign), that is, exactly when \(\mathrm{sgn}\left(\delta_{\varepsilon}\left(k\right)\right)\neq\mathrm{sgn}\left(\delta_{\varepsilon}\alpha\left(k\right)\right)\), as for (\ref{formula: faces sign}) and (\ref{formula: edges sign}).  Thus an even number of \(\eta_{0}\) appear in total, so their product is \(1\).

\paragraph{Value of the sign}

The product of sign indices in (\ref{formula: faces sign}) will have a term corresponding to \(\left|k\right|\in\left[n\right]\) if \(\varepsilon\left(k\right)=-1\).  Exactly one of \(k\) and \(-k\) is in \(\mathrm{FD}\left(\alpha\right)\) (call this one \(k\)) while \(-k=\delta\alpha\left(l\right)\) for some \(l\in\mathrm{FD}\left(\alpha\right)\).  When (\ref{formula: faces sign}) is written in terms of \(\eta_{0}\), \(\eta_{k}=\mathrm{sgn}\left(k\right)\eta_{0}\), and \(\eta_{-k}=\mathrm{sgn}\left(\alpha\left(l\right)\right)\eta_{0}=\mathrm{sgn}\left(k\right)\eta_{0}\), so the total contribution is a factor of \(1\).

Likewise, when (\ref{formula: edges sign}) is written in terms of \(\eta_{0}\), it will have a term corresponding to \(\left|k\right|\in\left[n\right]\) when the one of \(\left\{k,-k\right\}\) which is in \(\mathrm{FD}\left(\alpha\right)\) is negative (say \(-k\)).  Then \(k=\delta\alpha\left(l\right)\) for some \(l\in\mathrm{FD}\left(\alpha\right)\).  Then \(\eta_{-k}=-\eta_{0}\) and \(\eta_{k}=\mathrm{sgn}\left(\alpha\left(l\right)\right)\eta_{0}=-\eta_{0}\), so again the contribution is a factor of \(1\).

Finally, we consider the contribution of (\ref{formula: vertices sign}).  We calculate that if it is \(\varphi_{\mathrm{Re}+}^{-1}\delta_{\varepsilon}\left(k\right)\) and \(\varphi_{\mathrm{Re}-}^{-1}\delta_{\varepsilon}\alpha\left(k\right)\) which appear in \(\mathrm{FD}\left(\sigma_{\mathrm{Re}}\right)\) (the first pair mentioned in the previous section), then either's contribution to (\ref{formula: vertices sign}) will be \(\eta_{0}\), while if it is the second pair \(\delta\varphi_{\mathrm{Re}-}^{-1}\delta_{\varepsilon}\alpha\left(k\right)\) and \(\delta\varphi_{\mathrm{Re}+}^{-1}\delta_{\varepsilon}\left(k\right)\) which appear the contribution will be \(-\eta_{0}\).  Thus, we wish to consider \(k\in\mathrm{FD}\left(\alpha\right)\) for which \(\varphi_{\mathrm{Re}+}^{-1}\delta_{\varepsilon}\left(k\right)\notin\mathrm{FD}\left(\sigma_{\mathrm{Re}}\right)\); or, mapping the pair partition \(\delta\alpha\) under \(\varphi_{\mathrm{Re}+}^{-1}\delta_{\varepsilon}\) (to \(\delta\sigma_{\mathrm{Re}}^{-1}\)), pairs such that the distinguished element (the element of \(\mathrm{FD}\left(\alpha\right)\)) is not mapped to the distinguished element (the element of \(\mathrm{FD}\left(\sigma_{\mathrm{Re}}\right)\)).

By Corollary~\ref{corollary: parity}, the parity of the number of such pairs is the sign of a permutation mapping \(\delta\sigma_{\mathrm{Re}}^{-1}\) to itself, taking \(\varphi_{\mathrm{Re}+}^{-1}\delta_{\varepsilon}\left(\mathrm{FD}\left(\alpha\right)\right)\) to \(\mathrm{FD}\left(\sigma_{\mathrm{Re}}\right)\).  We may construct such a permutation in several steps.  Let \(\rho_{1}:=\varphi_{\mathrm{Re}+}\) (we calculate \(\mathrm{sgn}\left(\rho_{1}\right)=\left(-1\right)^{n-\#\left(\varphi\right)}\)), which maps \(\delta\sigma_{\mathrm{Re}}\) (distinguished elements \(\varphi_{\mathrm{Re}+}^{-1}\delta_{\varepsilon}\left(\mathrm{FD}\left(\alpha\right)\right)\)) to \(\delta\delta_{\varepsilon}\alpha\delta_{\varepsilon}\) (distinguished elements \(\delta_{\varepsilon}\left(\mathrm{FD}\left(\alpha\right)\right)\)).  Let \(\rho_{2}\) map \(\delta\delta_{\varepsilon}\alpha\delta_{\varepsilon}\) (same distinguished elements) to \(\delta\) (disinguished elements are positive integers); then \(\mathrm{sgn}\left(\rho_{2}\right)=\left(-1\right)^{\#\left(\alpha\right)/2+\left|\left[n\right]\cap\left(\delta_{\varepsilon}\left(\mathrm{FD}\left(\alpha\right)\right)\right)\right|}\).  Let \(\rho_{3}\) map \(\delta\) (same distinguished elements) to \(\delta\sigma_{\mathrm{Re}}^{-1}\) (distinguished elements are \(\mathrm{FD}\left(\sigma_{\mathrm{Re}}\right)\)); then \(\mathrm{sgn}\left(\rho_{3}\right)=\left(-1\right)^{\#\left(\sigma_{\mathrm{Re}}\right)/2+\left|\left[n\right]\cap\mathrm{FD}\left(\sigma_{\mathrm{Re}}\right)\right|}\).  Then the sign of their product is \(\left(-1\right)^{\chi\left(\varphi,\delta_{\varepsilon}\alpha\delta_{\varepsilon}\right)+\left|\left[n\right]\cap\left(\delta_{\varepsilon}\left(\mathrm{FD}\left(\alpha\right)\right)\right)\right|+\left|\left[n\right]\cap\mathrm{FD}\left(\sigma_{\mathrm{Re}}\right)\right|}\).

We must also keep track of the contribution of each pair.  If both or neither of \(\delta\varphi_{\mathrm{Re}-}^{-1}\delta_{\varepsilon}\alpha\left(k\right)\) and \(\delta\varphi_{\mathrm{Re}+}^{-1}\delta_{\varepsilon}\left(k\right)\) are negative (i.e.\ if the corresponding pair of \(\delta\sigma_{\mathrm{Re}}^{-1}\) have opposite signs) the pair contributes \(1\), so we must subtract the number of such pairs.  The parity of the number of pairs of \(\delta\sigma_{\mathrm{Re}}\) consisting of a positive and a negative integer whose distinguished element is switched by \(\rho\) will be the parity of \(\left|\left[n\right]\cap\left(\delta_{\varepsilon}\left(\mathrm{FD}\left(\alpha\right)\right)\right)\right|-\left|\left[n\right]\cap\mathrm{FD}\left(\sigma_{\mathrm{Re}}\right)\right|\).  The sign follows.
\end{proof}

We can adapt this expression to the expected value of a quaternion-valued expression, and similarly to a matrix-valued expression using the point at infinity introduced in Definition~\ref{definition: traces}:
\begin{proposition}
Let \(w:\left[n\right]\rightarrow\left[C\right]\), \(X_{1},\ldots,X_{n}\), and \(Y_{1},\ldots,Y_{n}\) be as in Proposition~\ref{proposition: topological expansion}.  Let \(\varphi_{\mathrm{Re}},\varphi_{\mathrm{tr}}\in S\left(\left[n\right]_{\infty}\right)\) and let \(\varepsilon:\left[n\right]\rightarrow\left\{1,-1\right\}\).  Then
\begin{multline*}
\mathbb{E}\left(\mathrm{Re}_{\varphi_{\mathrm{Re}}}\mathrm{tr}_{\varphi_{\mathrm{tr}}}\left(X_{1}Y_{1},\ldots,X_{n}Y_{n}\right)\right)
\\=\sum_{\substack{\alpha=\alpha_{1}\cdots\alpha_{n}\\\alpha_{c}\in\mathrm{PM}\left(w^{-1}\left(c\right)\right)}}\left(-2\right)^{\chi\left(\varphi_{\mathrm{Re}},\delta_{\varepsilon}\alpha\delta_{\varepsilon}\right)-2\#\left(\varphi_{\mathrm{Re}}\right)}N^{\chi\left(\varphi_{\mathrm{tr}},\delta_{\varepsilon}\alpha\delta_{\varepsilon}\right)-2\#\left(\varphi_{\mathrm{tr}}\right)}
\\\times f_{1}\left(\alpha_{1}\right)\cdots f_{n}\left(\alpha_{n}\right)\mathbb{E}\left(\mathrm{Re}_{\mathrm{FD}\left(K\left(\varphi_{\mathrm{Re}},\delta_{\varepsilon}\alpha\delta_{\varepsilon}\right)^{-1}\right)}\mathrm{tr}_{\mathrm{FD}\left(K\left(\varphi_{\mathrm{tr}},\delta_{\varepsilon}\alpha\delta_{\varepsilon}\right)^{-1}\right)}\left(Y_{1},\ldots,Y_{n}\right)\right)
\end{multline*}
(where \(\alpha\) is taken to act trivially on \(\pm\infty\), and its cycles are counted as a permutation on \(\pm\left[n\right]_{\infty}\)).
\end{proposition}

\begin{remark}
In all of the matrix models we consider, \(\lim_{N\rightarrow\infty}f\left(\alpha\right)\) exists for all \(\alpha\).  The order of a term in \(N\) then depends only on the Euler characteristic of \(\varphi_{\mathrm{tr}}\) with \(\alpha\).  Typically, this means that highest order terms will be the ones where \(\alpha\) satisfies the appropriate noncrossing conditions on \(\varphi_{\mathrm{tr}}\).  Intuitively, this is because a sphere has the highest Euler characteristic, and the graph of edges on a sphere may be drawn on a plane, without crossings.  See Definition~\ref{definition: planar} for an example of a definition of noncrossing.  See also \cite{MR2052516, MR2036721} for more on the role of noncrossing diagrams in random matrix theory, and \cite{MR1813436, MR2851244, MR3217665, 2012arXiv1204.6211R} for more on noncrossing maps on unoriented surfaces.
\end{remark}

\begin{example}
Let \(Z_{1}\), \(Z_{2}\) be independent ensembles of Ginibre matrices, let \(W=\frac{1}{N}G^{\ast}DG\) be a Wishart matrix, let \(U\) be a Haar-distributed symplectic matrix, and let \(Y_{1},\ldots,Y_{10}\) be quaternionic random matrices.  Let all these ensembles be independent.  We consider the calculation of the expression
\[\mathbb{E}\left[U^{\ast}Y_{1}\mathrm{Re}\left(Z_{1}Y_{2}\right)\mathrm{tr}\left(UY_{3}WY_{4}\right)\mathrm{Re}\left(\mathrm{tr}\left(Y_{5}UY_{6}Z_{1}Y_{7}UY_{8}\right)\mathrm{tr}\left(Z_{2}Y_{9}Z_{2}^{\ast}Y_{10}\right)\right)\right]\textrm{.}\]
Numbering the matrices from left to right (and taking an implicit identity matrix directly before \(Y_{5}\), the permutations associated with \(\mathrm{Re}\) and \(\mathrm{tr}\) are
\[\varphi_{\mathrm{Re}}=\left(\infty,1,3,4\right)\left(2\right)\left(5,6,7,8,9,10\right)\]
and
\[\varphi_{\mathrm{tr}}=\left(\infty,1,2,3,4\right)\left(5,6,7,8\right)\left(9,10\right)\textrm{,}\]
and \(\varepsilon\left(k\right)\) is equal to \(1\) on \(2\), \(3\), \(4\), \(5\), \(6\), \(7\), \(8\), \(9\) and \(-1\) on \(1\) and \(10\).  We calculate the term associated with
\begin{multline*}
\alpha=\left(1,-6,8,-3\right)\left(3,-8,6,-1\right)\left(2,-7\right)\left(7,-2\right)\left(4\right)\left(-4\right)\left(5\right)\left(-5\right)\\\left(9,-10\right)\left(10,-9\right)\textrm{,}
\end{multline*}
which we can see from Sections~\ref{section: Gaussian} and \ref{section: Haar} is the product of permutations associated with nonzero values.  We calculate that the contribution of both of the Ginibre matrices is \(1\), the contribution of the Wishart matrix is \(\mathrm{Re}\left(\mathrm{tr}\left(D\right)\right)\), and (from \cite{MR2217291, MR2567222}) the contribution of the \(U\) is
\[\mathrm{wg}\left(\Lambda\left(1,-6,8,-3\right)\right)=\mathrm{wg}\left(\left[2\right]\right)=\frac{\left(2N\right)^{2}}{\left(2N+1\right)\left(2N-2\right)}\textrm{.}\]

We calculate
\begin{multline*}
K\left(\varphi_{\mathrm{Re}},\delta_{\varepsilon}\alpha\delta_{\varepsilon}\right)=\left(\infty,4,3\right)\left(-3,-4,-\infty\right)\\\left(1,5,10,8,-6,2,-7\right)\left(7,-2,6,-8,-10,-5,-1\right)\left(9\right)\left(-9\right)
\end{multline*}
and
\begin{multline*}
K\left(\varphi_{\mathrm{tr}},\delta_{\varepsilon}\alpha\delta_{\varepsilon}\right)=\left(\infty,4,3\right)\left(-3,-4,-\infty\right)\left(1,5,8,-6\right)\left(6,-8,-5,-1\right)\\\left(2,-7\right)\left(7,-2\right)\left(9\right)\left(-9\right)\left(10\right)\left(-10\right)
\end{multline*}
which gives us \(\chi\left(\varphi_{\mathrm{Re}},\delta_{\varepsilon}\alpha\delta_{\varepsilon}\right)=3+5+3-11=0\) and \(\chi\left(\varphi_{\mathrm{tr}},\delta_{\varepsilon}\alpha\delta_{\varepsilon}\right)=3+5+5-11=-2\).  Thus the total contribution of this term is
\[\frac{\mathrm{Re}\left(\mathrm{tr}\left(D\right)\right)\mathbb{E}\left[Y_{3}Y_{4}\mathrm{Re}\left(\mathrm{tr}\left(Y_{1}\mathrm{tr}\left(Y_{7}^{\ast}Y_{2}\right)Y_{6}^{\ast}Y_{8}\mathrm{tr}\left(Y_{10}\right)Y_{5}\right)\right)\mathrm{Re}\left(\mathrm{tr}\left(Y_{9}\right)\right)\right]}{2^{6}N^{8}\left(2N+1\right)\left(2N-2\right)}\]
(where we have used Algorithm~\ref{algorithm: lub}) to write the part of the expression with the \(Y_{k}\) as a product with \(\mathrm{Re}\) and \(\mathrm{tr}\) applied to subexpressions.  See Figure~\ref{figure: gluing} for the topological constructions associated with \(\mathrm{Re}\) and \(\mathrm{tr}\).

\begin{figure}
\centering
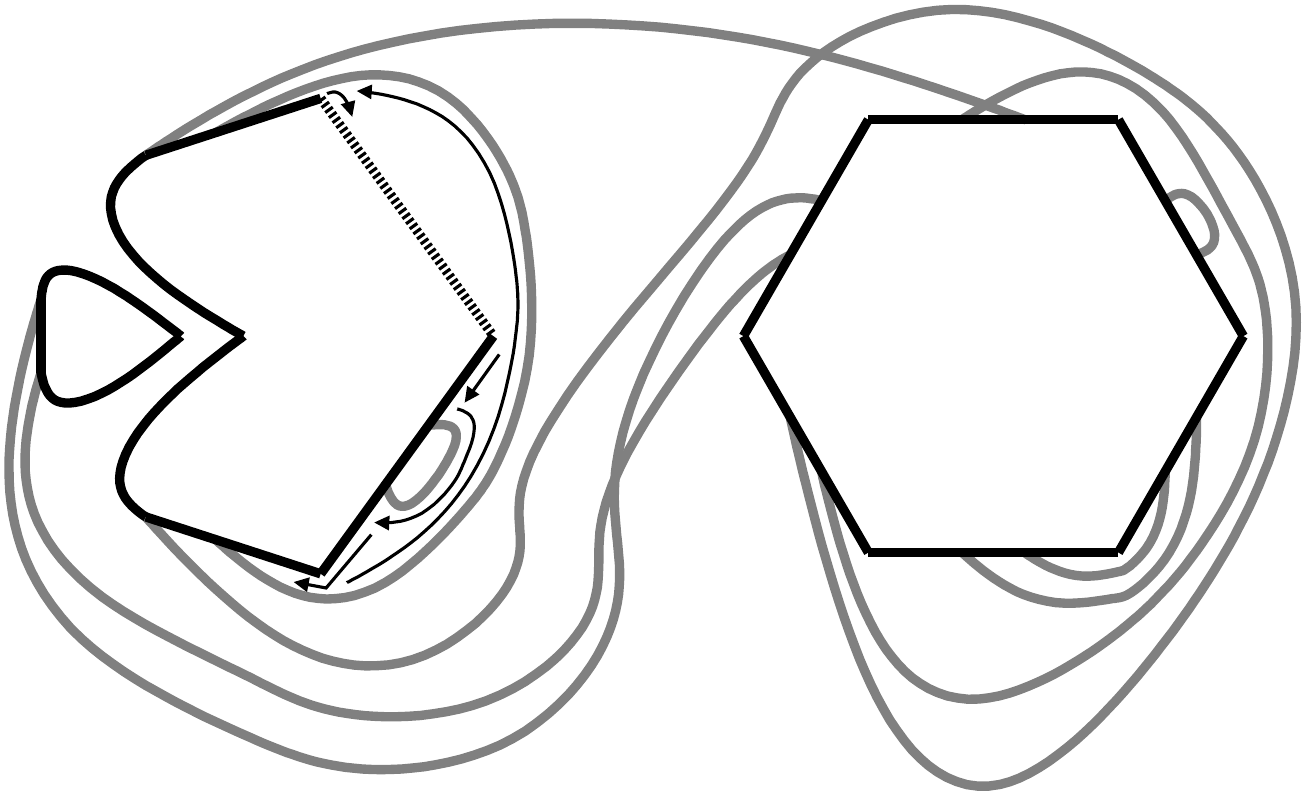
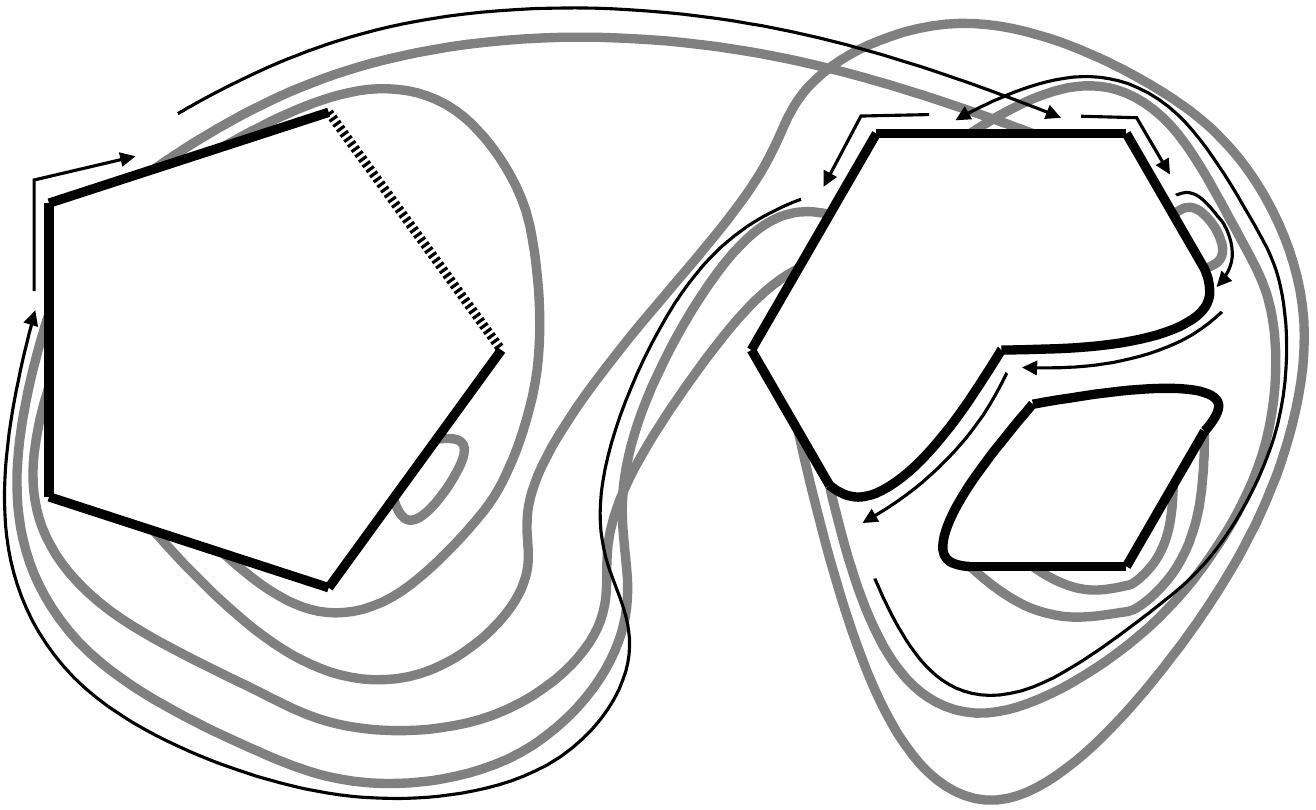
\caption{The faces representing the cycles of the real part (above) and trace (below) in Example~\ref{example: topological expansion} with the same hyperedge identifications.  The vertex \(\left(\infty,4,3\right)\) is traced by the arrows (above).  The dashed line represents the point at infinity.  The grey lines may be taken to represent hyperedge identifications or interpreted as identifying only that half of the edge.  The vertex \(\left(1,4,5,-6\right)\) is shown with arrows (below).}
\label{figure: gluing}
\end{figure}

\label{example: topological expansion}
\end{example}

\section{Bracket Diagrams}
\label{section: bracket diagrams}

For \(A\in M_{N\times N}\left(\mathbb{H}\right)\), we may interpret \(\mathrm{Re}\left(A\right)\) as the entrywise real part, and \(\mathrm{tr}\left(A\right)\), a quaternion, as a quaternion multiple of the identity matrix \(\mathrm{tr}\left(A\right)I_{n}\).  It is then possible to write some expressions \(\mathrm{Re}_{\varphi_{\mathrm{Re}}}\mathrm{tr}_{\varphi_{\mathrm{tr}}}\left(A_{1},\ldots,A_{n}\right)\) as a product of matrices with the functions \(\mathrm{Re}\) and \(\mathrm{tr}\) applied to bracketed subexpressions, where the bracket diagram is legal, as in the introduction and in Example~\ref{example: topological expansion}.

\begin{definition}
We define a {\em bracket diagram} on symbols \(X_{1},\ldots,X_{n}\) as a legal placement of left and right brackets (i.e., scanning left to right, the number of right brackets never exceeds the number of left brackets) between the symbols in some ordering of the symbols.

We associate a permutation \(\pi\in S\left(\left[n\right]_{\infty}\right)\) to a bracket diagram by placing a term \(X_{\infty}\) before the first term and after the last.  Then, for \(k\in\left[n\right]_{\infty}\), we let \(\pi\left(k\right)\) be the subscript of the symbol after \(X_{k}\), ignoring any bracketed intervals (that is, if \(X_{k}\) is immediately followed by a left bracket, then \(\pi\left(k\right)\) is the subscript of the first symbol following the partnered right bracket.
\end{definition}

If pairs of brackets are associated to two different functions (in our case \(\mathrm{Re}\) and \(\mathrm{tr}\)), then we may also consider the permutation associated to the bracket diagram that includes only the brackets associated with a given function.  In Example~\ref{example: topological expansion}, the permutation associated with the bracket diagram is \(\left(\infty,1\right)\left(2\right)\left(3,4\right)\left(5,6,7,8\right)\left(9,10\right)\).  The permutation associated with \(\mathrm{Re}\) is \(\varphi_{\mathrm{Re}}\), and the permutation associated with \(\mathrm{tr}\) is \(\varphi_{\mathrm{tr}}\).

We note that, in such an expression, if \(\varphi_{\mathrm{Re}}\) is the permutation associated with \(\mathrm{Re}\) and \(\varphi_{\mathrm{tr}}\) is the permutation associated with \(\mathrm{tr}\), then it is equal to \(\mathrm{Re}_{\varphi_{\mathrm{Re}}}\mathrm{tr}_{\varphi_{\mathrm{tr}}}\left(X_{1},\ldots,X_{n}\right)\): a real part is of the form \(A\otimes\delta_{\eta_{1}\eta_{2}}\) for some matrix \(A\in M_{N\times N}\left(\mathbb{C}\right)\) (and respectively a trace is of the form \(\delta_{\iota_{1}\iota_{2}}\otimes q\) for some \(q\in\mathbb{H}\)), so the \(\eta\) indices (resp.\ \(\iota\) indices) of the matrices before and after, cyclically within \(\mathrm{Re}\) (resp.\ \(\mathrm{tr}\)) brackets are constrained as they would be if the bracketed expression were not present.

\begin{definition}
We define a metric on \(S\left(I\right)\) by constructing its Cayley graph with the transpositions as generators.  In this graph, the vertices are elements of \(S\left(I\right)\), and two permutations are joined by an edge if one may be obtained from the other by multiplying by a transposition.  The distance between \(\pi,\rho\in S_{n}\) is the same as the distance from the identity \(e\) to \(\pi^{-1}\rho\), which is \(\left|I\right|-\#\left(\pi^{-1}\rho\right)\).  A {\em geodesic} is a minimal length path between two given permutations.

We construct a poset \(\preceq\) on \(S\left(I\right)\) as follows.  We note that (left or right) multiplying a permutation \(\pi\in S\left(I\right)\) by an involution \(\left(a,b\right)\in S\left(I\right)\) joins the cycles of \(\pi\) containing \(a\) and \(b\) if they are in different cycles, or splits the cycle containing \(a\) and \(b\) if they are in the same cycle of \(\pi\).  For \(\pi,\rho\in S\left(I\right)\), \(\pi\preceq\rho\) if \(\rho\) may be constructed from \(\pi\) by successively joining cycles of \(\pi\) in this manner.
\end{definition}

\begin{lemma}
If \(\pi,\rho\in S\left(\left[n\right]_{\infty}\right)\), then \(\pi\preceq\rho\) if and only if we may construct a bracket diagram with permutation \(\pi\) by adding brackets to a bracket diagram for \(\rho\).
\label{lemma: brackets}
\end{lemma}
\begin{proof}
Multiplying \(\rho\) by transposition \(\left(a,b\right)\) where \(a\) and \(b\) are in the same cycle of \(\rho\) splits that cycle into two intervals (this can be shown by calculation).  The corresponding bracketed interval (ignoring any bracketed intervals within it) will contain at least one of those two intervals of the cycle as an interval of its linear representation, which may be marked off by brackets.  This does not change which brackets are paired with which, and the cycle containing \(a\) and \(b\) has been split in the way that multiplying by \(\left(a,b\right)\) does.  We may continue in this way until we have obtained a diagram for \(\pi\).

Conversely, if a bracket diagram for \(\pi\) is constructed by adding brackets to a bracket diagram for \(\rho\), then ignoring a pair of these brackets corresponds to joining two cycles by multiplying by the appropriate transposition.
\end{proof}

The following folklore lemma may be interpreted as an upper limit on the Euler characteristic.  See \cite{MR2052516} for a proof.
\begin{lemma}
If \(\pi,\rho\in S\left(I\right)\), then
\[\#\left(\pi\right)+\#\left(\rho\right)+\#\left(\pi\rho\right)\leq\left|I\right|+2\#\left(\Pi\left(\pi\right)\vee\Pi\left(\rho\right)\right)\]
\label{lemma: triangle}
\end{lemma}
\begin{definition}
If \(\pi,\rho\in S\left(I\right)\), we say that \(\pi\) is {\em planar} on \(\rho\) if
\[\#\left(\pi\right)+\#\left(\rho\right)+\#\left(\pi\rho\right)-\left|I\right|=2\#\left(\Pi\left(\pi\right)\vee\Pi\left(\rho\right)\right)\textrm{.}\]
Intuitively, this is means that the Euler characteristic is \(2\) per connected component, that is, each component is a sphere.  (See \cite{MR1475837, MR0404045, MR1603700, MR2036721, MR2052516} for more on noncrossing permutations.  We use a slightly different convention than some of the references for which permutations are inverted, since we would like all elements---faces, hyperedges, and vertices---to be oriented counterclockwise.)

If \(\pi,\rho\in\mathrm{PM}\left(I\right)\), we say that \(\pi\) is planar on \(\rho\) if there is a \(J\subseteq\pm I\) such for all \(k\in I\) exactly one of \(k\) and \(-k\) is in \(J\), which is a union of cycles of \(\pi\) and \(\rho\), and such that \(\left.\pi\right|_{J}\) is planar on \(\left.\rho\right|_{J}\).  (Intuitively, the existence of such a \(J\) means that the hypermap is orientable, and the planarity condition means that its orientable two-sheeted covering space is a collection of spheres, which together imply that the original map is a collection of spheres.  See \cite{2012arXiv1204.6211R} for more on this construction.)
\label{definition: planar}
\end{definition}
The following lemma is from \cite{MR1475837}.  It shows that planarity is equivalent to not having crossings.
\begin{lemma}[Biane]
If \(\pi\) has a single cycle, then \(\rho\) is noncrossing on \(\pi\) exactly when \(\pi\) there are no \(a,b,c,d\in I\) such that \(\left.\pi\right|_{\left\{a,b,c,d\right\}}=\left(a,b,c,d\right)\) but \(\left.\rho\right|_{\left\{a,b,c,d\right\}}=\left(a,c\right)\left(b,d\right)\).
\label{lemma: noncrossing}
\end{lemma}

We present several equivalent conditions that allow two bracket diagrams to be constructed on the same symbols:
\begin{lemma}
Let \(\pi,\rho\in S\left(I\right)\).  Then the following are equivalent:
\begin{enumerate}
	\item \label{item: upper bound} \(\pi\) and \(\rho\) have an upper bound;
	\item \label{item: least upper bound} \(\pi\) and \(\rho\) have an upper bound \(\sigma\) such that \(\Pi\left(\pi\right)\vee\Pi\left(\rho\right)=\Pi\left(\sigma\right)\);
	\item \label{item: geodesic} \(\pi\) and \(\rho\) have an upper bound lying on a geodesic between \(\pi\) and \(\rho\);
	\item \label {item: noncrossing} \(\pi\) is planar on \(\rho^{-1}\).
\end{enumerate}
\label{lemma: upper bound}
\end{lemma}
\begin{proof}
To show \(\textrm{\ref{item: upper bound}}\Rightarrow\textrm{\ref{item: least upper bound}}\), we construct a permutation \(\tau\) by taking the product \(\left.\sigma\right|_{V_{1}}\cdots\left.\sigma\right|_{V_{k}}\) where \(\Pi\left(\pi\right)\vee\Pi\left(\rho\right)=\left\{V_{1},\ldots,V_{k}\right\}\).  We allow \(V\in\Pi\left(\pi\right)\vee\Pi\left(\rho\right)\) to inherit the brackets from \(\pi\) or \(\rho\) as well.  (To be precise, we consider the bracket diagram of \(\pi\) or \(\rho\) on \(\sigma\), then delete elements not in \(V\).  This may leave many trivial pairs of brackets, but the partner of each bracket will not change.)  Then any cycle of \(\pi\) or \(\rho\) appearing in this block of \(\tau\) is completely contained in this block, and is therefore a cycle of the permutation of the bracket diagram on this block.  If we combine all the blocks, we have a bracket diagram of \(\pi\) or \(\rho\) on \(\tau\).  By construction \(\tau\preceq\Pi\left(\pi\right)\vee\Pi\left(\rho\right)\), and since \(\tau\succeq\pi,\rho\), we know \(\tau\succeq\Pi\left(\pi\right)\vee\Pi\left(\rho\right)\).

To show \(\textrm{\ref{item: least upper bound}}\Rightarrow\textrm{\ref{item: geodesic}}\), we represent \(\pi\) and \(\rho\) as bracket diagrams on \(\sigma\).  We construct a path from \(\pi\) (resp.\ \(\rho\) )to \(\sigma\) by removing brackets (joining cycles) from the diagram of \(\pi\) (resp.\ \(\rho\)) on \(\sigma\) until we have \(\sigma\).  The length of the first path is \(\#\left(\pi\right)-\#\left(\sigma\right)\) and the length of the second is \(\#\left(\rho\right)-\#\left(\sigma\right)\).  The paths together form a geodesic, since the distance from \(\pi\) to \(\rho\) is \(\left|I\right|-\#\left(\rho\pi^{-1}\right)\), which by Lemma~\ref{lemma: triangle} is greater than or equal to \(\#\left(\pi\right)+\#\left(\rho\right)-2\#\left(\Pi\left(\pi\right)\vee\Pi\left(\rho\right)\right)=\#\left(\pi\right)+\#\left(\rho\right)-2\#\left(\sigma\right)\).

For \(\textrm{\ref{item: geodesic}}\Rightarrow\textrm{\ref{item: noncrossing}}\), let \(\tau\) be an upper bound on a geodesic.  The length of the geodesic is \(n-\#\left(\rho\pi^{-1}\right)=\left[\#\left(\pi\right)-\#\left(\tau\right)\right]+\left[\#\left(\rho\right)-\#\left(\tau\right)\right]\), so \(\#\left(\pi\right)+\#\left(\rho\right)+\#\left(\rho\pi^{-1}\right)-n=2\#\left(\tau\right)\), so \(\pi\) is planar on \(\rho^{-1}\).

For \(\textrm{\ref{item: noncrossing}}\Rightarrow\textrm{\ref{item: upper bound}}\), we use Algorithm~\ref{algorithm: lub} to construct a permutation \(\sigma\) such that \(\pi\) and \(\rho\) are both planar on \(\sigma^{-1}\) and such that \(\Pi\left(\sigma\right)\succeq\Pi\left(\pi\right),\Pi\left(\rho\right)\).  Then we may construct a bracket diagram for \(\pi\) (resp.\ \(\rho\)) on each cycle of \(\sigma\) by placing a bracket before the first element of any cycle of \(\pi\) and after the last element.  These brackets remain paired in this way: if brackets were not properly nested, then the cycles of \(\pi\) (resp.\ \(\rho\)) must exhibit the crossing pattern in Lemma~\ref{lemma: noncrossing}.
\end{proof}

In addition, it is also necessary that the two bracket diagrams can be written simultaneously.

\begin{proposition}
For \(\varphi_{\mathrm{Re}},\varphi_{\mathrm{tr}}\in S\left(\left[n\right]_{\infty}\right)\), there is a bracket diagram equal to \(\mathrm{Re}_{\varphi_{\mathrm{Re}}}\mathrm{tr}_{\varphi_{\mathrm{tr}}}\left(X_{1},\ldots,X_{n}\right)\) if and only if \(\varphi_{\mathrm{Re}}\) and \(\varphi_{\mathrm{tr}}\) satisfy Lemma~\ref{lemma: upper bound} and
\begin{multline}
\#\left(\Pi\left(\varphi_{\mathrm{Re}}\right)\wedge\Pi\left(\varphi_{\mathrm{tr}}\right)\right)-\#\left(\Pi\left(\varphi_{\mathrm{Re}}\right)\vee\Pi\left(\varphi_{\mathrm{tr}}\right)\right)
\\=\left(\#\left(\Pi\left(\varphi_{\mathrm{Re}}\right)\wedge\Pi\left(\varphi_{\mathrm{tr}}\right)\right)-\#\left(\varphi_{\mathrm{Re}}\right)\right)\\+\left(\#\left(\Pi\left(\varphi_{\mathrm{Re}}\right)\wedge\Pi\left(\varphi_{\mathrm{tr}}\right)\right)-\#\left(\varphi_{\mathrm{tr}}\right)\right)\textrm{.}
\label{formula: glb}
\end{multline}

If \(\varphi_{\mathrm{Re}},\varphi_{\mathrm{tr}}\in\mathrm{PM}\left(\left[n\right]_{\infty}\right)\), then it is possible to find a bracket diagram with functions \(\mathrm{Re}\) and \(\mathrm{tr}\) applied to bracketed intervals where \(\varphi_{\mathrm{Re}}\) and \(\varphi_{\mathrm{tr}}\) are the permutations associated to the functions \(\mathrm{Re}\) and \(\mathrm{tr}\) respectively exactly when \(\varphi_{\mathrm{Re}}\) is planar on \(\varphi_{\mathrm{tr}}\) and they satisfy (\ref{formula: glb}).
\label{proposition: bracket diagram}
\end{proposition}
\begin{proof}
If there is a bracket diagram equal to \(\mathrm{Re}_{\varphi_{\mathrm{Re}}}\mathrm{tr}_{\varphi_{\mathrm{tr}}}\left(X_{1},\ldots,X_{n}\right)\), then the order in which the symbols appear (with \(X_{\infty}\) at the beginning) is the cycle notation of an upper bound for \(\varphi_{\mathrm{Re}}\) and \(\varphi_{\mathrm{tr}}\), so the two permutations satisfy Lemma~\ref{lemma: upper bound}.  Furthermore, the brackets of both functions must be legal, so if we add each pair of brackets one by one to the cycle notation of upper bound \(\zeta\) (where \(\zeta\) is chosen so that \(\Pi\left(\zeta\right)=\Pi\left(\varphi_{\mathrm{Re}}\right)\vee\Pi\left(\varphi_{\mathrm{tr}}\right)\)), each pair will split exactly one cycle.  Thus the number by which the number of cycles increases when we construct both \(\varphi_{\mathrm{Re}}\) and \(\varphi_{\mathrm{tr}}\) on the cycle notation of \(\zeta\) should be the sum of the numbers by which the number of cycles increases when we construct \(\varphi_{\mathrm{Re}}\) on \(\zeta\) and when we construct \(\varphi_{\mathrm{tr}}\) on \(\zeta\), so the two permutations satisfy (\ref{formula: glb}).

Conversely, if \(\varphi_{\mathrm{Re}}\) and \(\varphi_{\mathrm{tr}}\) satisfy Lemma~\ref{lemma: upper bound} and (\ref{formula: glb}), then it is possible to find a permutation \(\zeta\) such that \(\varphi_{\mathrm{Re}},\varphi_{\mathrm{tr}}\preceq\zeta\).  It is then possible to construct both \(\varphi_{\mathrm{Re}}\) and \(\varphi_{\mathrm{tr}}\) as bracket diagrams on the cycle notation of \(\zeta\).  If it is not possible to do so such that the brackets are properly nested, then whenever we construct \(\varphi_{\mathrm{Re}}\) on \(\zeta\), then add the brackets for a construction of \(\varphi_{\mathrm{tr}}\), at least one pair of brackets from \(\varphi_{\mathrm{tr}}\) must split two cycles of the existing permutation.  Then the number by which \(\#\left(\Pi\left(\varphi_{\mathrm{Re}}\right)\wedge\Pi\left(\varphi_{\mathrm{tr}}\right)\right)\) exceeds \(\#\left(\zeta\right)\) is greater than the sum of the number by which \(\#\left(\varphi_{\mathrm{Re}}\right)\) exceeds \(\#\left(\zeta\right)\) plus the number by which \(\#\left(\varphi_{\mathrm{tr}}\right)\) exceeds \(\#\left(\zeta\right)\), i.e.\ the permutations do not satisfy (\ref{formula: glb}).

For the second part of the proposition, we observe that if \(\varphi_{\mathrm{Re}}\) is planar on \(\varphi_{\mathrm{tr}}\), there must be an \(I\subseteq\pm\left[n\right]_{\infty}\) such that exactly one of \(k\) and \(-k\) is in \(I\) for all \(k\in\left[n\right]_{\infty}\) which is a union of cycles of both \(\varphi_{\mathrm{Re}}\) and \(\varphi_{\mathrm{tr}}\), so we may repeat the above proof on their restrictions to \(I\).  Conversely, if such a bracket diagram exists, then there must be such an \(I\), and again we may repeat the above proof.
\end{proof}

We adapt the following algorithm from \cite{MR0404045} (pages 41--42):
\begin{algorithm}
Let \(\varphi_{\mathrm{Re}},\varphi_{\mathrm{tr}}\in\mathrm{PM}\left(\left[n\right]_{\infty}\right)\) such that \(\varphi_{\mathrm{Re}}\) is planar on \(\varphi_{\mathrm{tr}}^{-1}\).  We construct a \(\zeta\in S\left(I\right)\) for some \(I\subseteq\pm\left[n\right]_{\infty}\) such that exactly one of \(\pm k\) is in \(I\) for all \(k\in\left[n\right]\).

We begin with symbol \(\infty\).  We then repeat the following, where \(k\) is the last chosen element:
\begin{itemize}
	\item If \(\varphi_{\mathrm{Re}}\left(k\right)\) is not in the same orbit of \(\varphi_{\mathrm{tr}}\) as any element already chosen, we let \(\zeta\left(k\right)=\varphi_{\mathrm{Re}}\left(k\right)\).
	\item If \(\varphi_{\mathrm{Re}}\left(k\right)\) is in the same orbit of \(\varphi_{\mathrm{tr}}\) as some element already chosen, we let \(\zeta\left(k\right)\) be the first of \(\varphi_{\mathrm{tr}}\left(k_{m}\right),\ldots,\varphi_{\mathrm{tr}}\left(k_{1}\right)\) which has not yet appeared, where \(k_{1},\ldots,k_{m}\) are the symbols in order which have already appeared.
	\item If all of \(\varphi_{\mathrm{tr}}\left(k_{m}\right),\ldots,\varphi_{\mathrm{tr}}\left(k_{1}\right)\) have all already appeared, we begin with an arbitrarily chosen \(k\) such that neither \(k\) or \(-k\) has appeared so far.  (We let this begin a new cycle of \(\zeta\) if we want \(\zeta\) to be a least upper bound, or we let it be \(\zeta\left(k\right)\) if we want \(\zeta\) to be cyclic.)
\end{itemize}
\label{algorithm: lub}
\end{algorithm}
Since we are always choosing an element from the same cycle as an element which has already appeared (and we exhaust all such cycles, since the second possibility exhausts cycles of \(\varphi_{\mathrm{tr}}\) and the first exhausts cycles of \(\varphi_{\mathrm{Re}}\) if the cycle of \(\varphi_{\mathrm{tr}}\) has not yet appeared), or such that its negative has not appeared, the domain of \(\zeta\) will be a union of cycles of both \(\varphi_{\mathrm{Re}}\) and \(\varphi_{\mathrm{tr}}\) which contains exactly one of \(k\) and \(-k\) for each \(k\in\pm\left[n\right]_{\infty}\).

See \cite{MR0404045}, pages 41--44 for a proof that both \(\varphi_{\mathrm{Re}}\) and \(\varphi_{\mathrm{tr}}\) are noncrossing, that is, planar, on \(\zeta^{-1}\).  By Lemma~\ref{lemma: upper bound}, they can then both be written as bracket diagrams on its cycle notation, and by Proposition~\ref{proposition: bracket diagram} it must be possible to do so legally as long as \(\varphi_{\mathrm{Re}}\) and \(\varphi_{\mathrm{tr}}\) satisfy (\ref{formula: glb}).

\begin{remark}
It is not always possible to write \(\mathrm{Re}_{K\left(\varphi_{\mathrm{Re}},\alpha\right)^{-1}}\) and \(\mathrm{tr}_{K\left(\varphi_{\mathrm{tr}},\alpha\right)^{-1}}\) simultaneously, even when it is possible to write \(\varphi_{\mathrm{Re}}\) and \(\varphi_{\mathrm{tr}}\) simultaneously (as was done in Example~\ref{example: topological expansion}).  A slightly larger class of permutations could be written simultaneously if we considered the transpose and entrywise conjugate separately.  Intuitively, this would allow permutations which do not have actual crossings, but which are ``nonstandard'', that is, a cycle is in the wrong order relative to the other permutation (see \cite{MR2052516} for this for formal definitions of this terminology).  To do so, we could create a four-sheeted covering space of the space, in which the original two-sheeted covering space is covered by disjoint copies with opposite orientations.

However, it would still be impossible to write simultaneously permutations which correspond to an actual crossing, such as \(\left(1,2,3,4\right)\left(-4,-3,-2,-1\right)\) and \(\left(1,3\right)\left(-3,-1\right)\left(2,4\right)\left(-4,-2\right)\), even though they can be produced as \(K\left(\varphi_{\mathrm{Re}},\alpha\right)^{-1}\) and \(K\left(\varphi_{\mathrm{tr}},\alpha\right)^{-1}\) where \(\varphi_{\mathrm{Re}}=\left(1,4,3,2\right)\), \(\varphi_{\mathrm{tr}}=\left(1\right)\left(2\right)\left(3\right)\left(4\right)\), and \(\alpha=\left(1,3\right)\left(-3,-1\right)\left(2,4\right)\left(-4,-2\right)\).
\end{remark}

\section{Matrices Constructed from Gaussian Random variables}
\label{section: Gaussian}

The following lemma for computing the expected value of products of Gaussian random variables is known as the Wick formula.
\begin{lemma}[Wick]
Let \(\xi_{1},\ldots,\xi_{n}\) be components of a centred multivariate Gaussian random variable.  Then
\[\mathbb{E}\left(\xi_{1},\ldots,\xi_{n}\right)=\sum_{\pi\in{\cal P}_{2}\left(n\right)}\prod_{\left\{k,l\right\}\in\pi}\mathbb{E}\left(\xi_{k}\xi_{l}\right)\textrm{.}\]
\label{lemma: Wick}
\end{lemma}
See, for example, \cite{MR2036721}, Chapter~3, for a proof.  We note that, if \(Z\) is a complex Gaussian random variable, we calculate that \(\mathbb{E}\left(Z^{2}\right)=0\).  A nonzero expected value must pair variables with their complex conjugates.

\begin{definition}
A {\em standard quaternionic Gaussian random variable} is
\[\xi_{0}+i\xi_{1}+j\xi_{2}+k\xi_{3}\]
where the \(\xi_{i}\) are independent \(N\left(0,\frac{1}{4}\right)\) random variables; or, in matrix form:
\[\left(\begin{array}{cc}\xi_{0}+i\xi_{1}&\xi_{2}+i\xi_{3}\\-\xi_{2}+i\xi_{3}&\xi_{0}-i\xi_{1}\end{array}\right)\textrm{.}\]
\end{definition}

We note that the expected value of the product of two entries (as appears in a Wick product) vanishes unless the entries have opposite indices.  If the entries are the diagonal entries, the expected value of the product is \(\frac{1}{2}\), and if they are the off-diagonal entries, the expected value of the product is \(-\frac{1}{2}\).

\subsection{Ginibre matrices}

\begin{definition}
Let \(G\) be an \(N\times N\) matrix whose entries are independent standard quaternionic Gaussian random variables.  Then \(Z:=\frac{1}{\sqrt{N}}G\) is a {\em quaternionic Ginibre} matrix.
\end{definition}

Ginibre matrices are originally defined in \cite{MR0173726}.

\begin{proposition}
Let \(Z\) be a Ginibre matrix.  Then
\begin{multline*}
\mathbb{E}\left(\prod_{k=1}^{n}Z_{\iota_{k}\iota_{-k};\eta_{k}\eta_{-k}}\right)
\\=\sum_{\substack{\alpha\in\mathrm{PM}\left(\pm\left[n\right]\right)\\\iota=\iota\circ\delta\alpha\\\mathrm{sgn}\cdot\eta=\mathrm{sgn}\circ\alpha\cdot\eta\circ\delta\alpha}}\left[\prod_{\substack{k\in\left[n\right]\\-k\in\mathrm{FD}\left(\alpha\right)}}\eta_{k}\eta_{-k}\right]\left(2N\right)^{\#\left(\alpha\right)/2-n}f\left(\alpha\right)
\end{multline*}
where \(f:\mathrm{PM}\left(n\right)\rightarrow\mathbb{C}\) is defined by
\[f\left(\alpha\right)=\left\{\begin{array}{ll}1,&\alpha\in{\cal P}_{2}\left(\pm\left[n\right]\right)\cap S_{\mathrm{alt}}\left(\pm\left[n\right]\right)\\0,&\textrm{otherwise}\end{array}\right.\textrm{.}\]
\label{proposition: Ginibre}
\end{proposition}
\begin{proof}
By the Wick formula (Lemma~\ref{lemma: Wick}),
\[\mathbb{E}\left(\prod_{k=1}^{n}Z_{\iota_{k}\iota_{-k};\eta_{k}\eta_{-k}}\right)=\frac{1}{N^{n/2}}\sum_{\pi\in{\cal P}_{2}}\prod_{\left\{k,l\right\}\in\pi}\mathbb{E}\left(G_{\iota_{k}\iota_{-k};\eta_{k}\eta_{-k}}G_{\iota_{l}\iota_{-l};\eta_{l}\eta_{-l}}\right)\textrm{.}\]
We note that the map \(\pi\mapsto\pi\delta\pi\) is a bijection between \({\cal P}_{2}\left(n\right)\) and \(\mathrm{PM}\left(n\right)\cap{\cal P}_{2}\left(\pm\left[n\right]\right)\cap S_{\mathrm{alt}}\left(\pm\left[n\right]\right)\).  (If pair \(\left\{k,l\right\}\) appears in pair \(\pi\), then pairs \(\left\{k,-l\right\}\) and \(\left\{-k,l\right\}\) appear in \(\pi\delta\pi\), from which this follows.)  We can then rewrite the above as a sum over \(\pi\mapsto\pi\delta\pi\):
\[\frac{1}{N^{n/2}}\sum_{\alpha\in\mathrm{PM}\left(n\right)\cap{\cal P}_{2}\left(n\right)\cap S_{\mathrm{alt}}\left(\pm\left[n\right]\right)}\prod_{\left\{k,-l\right\}\in\mathrm{FD}\left(\alpha\right)}\mathbb{E}\left(G_{\iota_{k}\iota_{-k};\eta_{k}\eta_{-k}}G_{\iota_{l}\iota_{-l};\eta_{l}\eta_{-l}}\right)\textrm{.}\]
The expected value term corresponding to \(k\) vanishes (and hence so does the entire product) unless the two terms are complex conjugates of the same random variable, i.e.\ unless \(\iota_{\pm k}=\iota_{\pm\delta\alpha\left(k\right)}\) and \(\eta_{k}=-\eta_{\delta\alpha\left(k\right)}\) (that is, \(\mathrm{sgn}\left(k\right)\eta_{k}=\mathrm{sgn}\left(\alpha\left(k\right)\right)\eta_{\alpha\delta\left(k\right)}\)).  The magnitude will then be \(\frac{1}{2}\), and the sign depends on whether the entries are from the diagonals or the off-diagonals, i.e., on \(\eta_{l}\eta_{-l}\).  The result follows.
\end{proof}

\begin{remark}
The \(1\times 1\) case of Proposition~\ref{proposition: Ginibre} or Proposition~\ref{proposition: topological expansion} with Ginibre matrices, or their linear extensions, may be considered forms of quaternionic Wick formulae.  See also \cite{MR2480549, MR2986851}.
\end{remark}

\subsection{Gaussian symplectic ensemble matrices}

\begin{definition}
Let \(G\) be an \(N\times N\) random matrix whose entries are independent standard quaternionic Gaussian random variables.  Let \(T:=\frac{1}{\sqrt{2N}}\left(G+G^{\ast}\right)\).  Then \(T\) is a Gaussian symplectic ensemble, or GSE, matrix.
\end{definition}

\begin{proposition}
Let \(T\) be a GSE matrix.  Then
\[\mathbb{E}\left(\prod_{k=1}^{n}T_{\iota_{k}\iota_{-k};\eta_{k}\eta_{-k}}\right)=\sum_{\substack{\alpha\in\mathrm{PM}\left(\pm\left[n\right]\right)\\\iota=\iota\circ\delta\alpha\\\mathrm{sgn}\cdot\eta=\mathrm{sgn}\circ\alpha\cdot\eta\circ\delta\alpha}}\left[\prod_{\substack{k\in\left[n\right]\\-k\in\mathrm{FD}\left(\alpha\right)}}\eta_{k}\eta_{-k}\right]\left(2N\right)^{\#\left(\alpha\right)/2-n}f\left(\alpha\right)\]
where \(f:\mathrm{PM}\left(N\right)\rightarrow\mathbb{C}\) is defined by
\[f\left(\alpha\right)=\left\{\begin{array}{ll}1\textrm{,}&\alpha\in{\cal P}_{2}\left(\pm\left[n\right]\right)\\0\textrm{,}&\textrm{otherwise}\end{array}\right.\textrm{.}\]
\end{proposition}
\begin{proof}
We note (using (\ref{formula: annoying sign})) that
\begin{align*}
T_{\iota_{k}\iota_{-k};\eta_{k}\eta_{-k}}&=\frac{1}{\sqrt{2N}}\left(G_{\iota_{k}\iota_{k};\eta_{k},\eta_{-k}}+G^{\ast}_{\iota_{k}\iota_{-k};\eta_{k}\eta_{-k}}\right)
\\&=\frac{1}{\sqrt{2N}}\left(G_{\iota_{k},\iota_{-k};\mathrm{sgn}\left(k\right)\eta_{k},\mathrm{sgn}\left(k\right)\eta_{-k}}\right.
\\&\left.+\eta_{k}\eta_{-k}G_{\iota_{-k},\iota_{k};\mathrm{sgn}\left(-k\right)\eta_{-k},\mathrm{sgn}\left(-k\right)\eta_{k}}\right)
\end{align*}
Expanding, we express the product as a sum over all \(\varepsilon:\left[n\right]\rightarrow\left\{1,-1\right\}\) (and defining \(\delta_{\varepsilon}\) as usual):
\begin{multline}
\prod_{k=1}^{n}T_{\iota_{k}\iota_{-k};\eta_{k}\eta_{-k}}=\left(2N\right)^{n/2}\sum_{\varepsilon:\left[n\right]\rightarrow\left\{1,-1\right\}}\left[\prod_{k\in\varepsilon^{-1}\left(\left[n\right]\right)}\eta_{k}\eta_{-k}\right]
\\\times\prod_{k=1}^{n}G_{\iota_{\delta_{\varepsilon}\left(k\right)},\iota_{-\delta_{\varepsilon}\left(k\right)};\varepsilon\left(k\right)\eta_{\delta_{\varepsilon}\left(k\right)},\varepsilon\left(k\right)\eta_{-\delta_{\varepsilon}\left(k\right)}}\textrm{.}
\label{formula: GSE expansion}
\end{multline}
By the Wick formula, the expected value of the product of Gaussians is:
\begin{multline*}
\mathbb{E}\left(\prod_{k\in\left[n\right]}G_{\iota_{\delta_{\varepsilon}\left(k\right)},\iota_{-\delta_{\varepsilon}\left(k\right)};\varepsilon\left(k\right)\eta_{\delta_{\varepsilon}\left(k\right)},\varepsilon\left(k\right)\eta_{-\delta_{\varepsilon}\left(k\right)}}\right)
\\=\sum_{\pi\in{\cal P}_{2}\left(n\right)}\prod_{\left\{k,l\right\}\in\pi}\mathbb{E}\left(G_{\iota_{\delta_{\varepsilon}\left(k\right)},\iota_{-\delta_{\varepsilon}\left(k\right)};\varepsilon\left(k\right)\eta_{\delta_{\varepsilon}\left(k\right)},\varepsilon\left(k\right)\eta_{-\delta_{\varepsilon}\left(k\right)}}\right.\\\left.\times G_{\iota_{\delta_{\varepsilon}\left(l\right)},\iota_{-\delta_{\varepsilon}\left(l\right)};\varepsilon\left(l\right)\eta_{\delta_{\varepsilon}\left(l\right)},\varepsilon\left(l\right)\eta_{-\delta_{\varepsilon}\left(l\right)}}\right)\textrm{.}
\end{multline*}
For a given \(\pi\), the product vanishes unless for each \(\left\{k,l\right\}\in\pi\), \(\iota_{\pm\delta_{\varepsilon}\left(k\right)}=\iota_{\pm\delta_{\varepsilon}\left(l\right)}\) and \(\varepsilon\left(k\right)\eta_{\pm\delta_{\varepsilon}\left(k\right)}=-\varepsilon\left(l\right)\eta_{\pm\delta_{\varepsilon}\left(l\right)}\); that is, unless \(\iota=\iota\circ\delta\alpha\) and \(\mathrm{sgn}\cdot\eta=\mathrm{sgn}\circ\alpha\cdot\eta\circ\delta\alpha\) where \(\alpha:=\delta_{\varepsilon}\delta\pi\delta\pi\delta_{\varepsilon}\delta\).

If pair \(\left\{k,l\right\}\in\pi\), then pairs \(\left\{\delta_{\varepsilon}\left(k\right),-\delta_{\varepsilon}\left(l\right)\right\}\) and \(\left\{-\delta_{\varepsilon}\left(k\right),\delta_{\varepsilon}\left(l\right)\right\}\) appear in \(\alpha\).  The contribution of the pair is \(\eta_{k}\eta_{-k}\frac{1}{2}\).  Including the product of \(\eta_{k}\eta_{-k}\) in (\ref{formula: GSE expansion}), this pair will contribute an odd number of \(\eta_{k}\eta_{-k}\) exactly when \(\varepsilon\left(k\right)=\varepsilon\left(l\right)\), that is, when \(\mathrm{sgn}\left(k\right)\neq\mathrm{sgn}\left(\alpha\left(k\right)\right)\), that is, when exactly one element of the corresponding pair in \(\mathrm{FD}\left(\alpha\right)\) is negative, as desired.

Different \(\varepsilon\) and \(\pi\) will produce same \(\alpha\) if and only if \(\varepsilon\left(k\right)\varepsilon\left(l\right)\) is the same for all \(\left\{k,l\right\}\in\pi\), so the map \(\left(\varepsilon,\pi\right)\mapsto\alpha:{\cal P}_{2}\left(n\right)\rightarrow\mathrm{PM}\left(n\right)\cap{\cal P}_{2}\left(\pm\left[n\right]\right)\) is \(2^{n/2}\) to one and onto.  This gives us the desired contribution for each \(\alpha\in\mathrm{PM}\left(N\right)\cap{\cal P}_{2}\left(\pm\left[n\right]\right)\).
\end{proof}

\subsection{Wishart matrices}

\begin{definition}
Let \(G\) be an \(M\times N\) quaternionic matrix with each \(G_{\kappa\iota}\) (\(\kappa,\iota\in\left[n\right]\)) an independent standard quaternionic Gaussian random variable, and let \(D\in M_{M\times M}\left(\mathbb{H}\right)\).  Then \(W:=\frac{1}{N}G^{\ast}DG\) is a {\em quaternionic Wishart matrix}.

It is often useful to consider sets of Wishart matrices with the same (i.e.\ not independent) \(G\) but with different \(D\): \(W_{k}:=\frac{1}{N}G^{\ast}D_{k}G\) for \(D_{k}\in M_{M\times M}\left(\mathbb{H}\right)\), \(k=1,2,\ldots\).
\end{definition}

\begin{proposition}
For \(k\in\left[n\right]\), let \(W_{k}=\frac{1}{N}G^{\ast}D_{k}G\) be Wishart matrices.  Then
\begin{multline*}
\mathbb{E}\left(W^{\left(1\right)}_{\iota_{1}\iota_{-1};\eta_{1}\eta_{-1}}\cdots W^{\left(n\right)}_{\iota_{n}\iota_{-n};\eta_{n}\eta_{-n}}\right)\\=\sum_{\substack{\alpha\in PM\left(n\right)\\\iota=\iota\circ\delta\alpha\\\mathrm{sgn}\cdot\eta=\mathrm{sgn}\circ\alpha\cdot\eta\circ\delta\alpha}}\left(2N\right)^{\#\left(\alpha\right)/2-n}\left[\prod_{\substack{k\in\left[n\right]\\-k\in\mathrm{FD}\left(\alpha\right)}}\eta_{k}\eta_{-k}\right]f\left(\alpha\right)
\end{multline*}
where \(f:\mathrm{PM}\left(n\right)\rightarrow\mathbb{C}\) is given by
\[f\left(\alpha\right)=\mathrm{Re}_{\mathrm{FD}\left(\alpha^{-1}\right)}\mathrm{tr}_{\mathrm{FD}\left(\alpha^{-1}\right)}\left(D_{1},\ldots,D_{n}\right)\textrm{.}\]
\end{proposition}
\begin{remark}
We are using conventions for \(\varphi\), \(\alpha\), and \(K\left(\varphi,\alpha\right)\) such that, when they are thought of as the faces, hyperedges, and vertices of a map, they are all oriented counter-clockwise from within.  We have defined \(f\left(\alpha\right)\) to be consistent with these conventions.  This is why we have chosen to define it in a way that it is \(\alpha^{-1}\) and not \(\alpha\) which appears.  (See comment in Definition~\ref{definition: planar}).
\end{remark}
\begin{proof}
Expanding each entry of a Wishart matrix into entries of \(G^{\ast}\), \(D_{k}\), and \(G\) (and applying (\ref{formula: annoying sign}) to the \(G^{\ast}\) in each Wishart matrix), we get
\begin{multline*}
\mathbb{E}\left(W^{\left(1\right)}_{\iota_{1}\iota_{-1};\eta_{1}\eta_{-1}}\cdots W^{\left(n\right)}_{\iota_{n}\iota_{-n};\eta_{n}\eta_{-n}}\right)
\\=\frac{1}{N^{n}}\sum_{\substack{\kappa:\pm\left[n\right]\rightarrow\left[M\right]\\\theta:\pm\left[n\right]\rightarrow\left\{1,-1\right\}}}\mathbb{E}\left(\prod_{k=1}^{n}\eta_{k}\theta_{k}G_{\kappa_{k},\iota_{k};-\theta_{k},-\eta_{k}}D_{\kappa_{k},\kappa_{-k};\theta_{k},\theta_{-k}}G_{\kappa_{-k},\iota_{-k};\theta_{-k}\eta_{-k}}\right)\textrm{.}
\end{multline*}

For fixed \(\kappa\) and \(\theta\), we apply Lemma~\ref{lemma: Wick} to the Gaussian factors (indexing by elements of \(\mathrm{PM}\left(\pm\left[n\right]\right)\) instead of the pairings on \(\pm\left[n\right]\) using the bijection in Lemma~\ref{lemma: pairing bijection}):
\begin{multline*}
\mathbb{E}\left(\prod_{k\in\pm\left[n\right]}G_{\kappa_{k},\iota_{k};-\mathrm{sgn}\left(k\right)\theta_{k},-\mathrm{sgn}\left(k\right)\eta_{k}}\right)
\\=\sum_{\alpha\in\mathrm{PM}\left(n\right)}\prod_{k\in\mathrm{FD}\left(\alpha\right)}\mathbb{E}\left(G_{\kappa_{k},\iota_{k};-\mathrm{sgn}\left(k\right)\theta_{k},-\mathrm{sgn}\left(k\right)\eta_{k}}\right.\\\left.G_{\kappa_{\delta\alpha\left(k\right)},\iota_{\delta\alpha\left(k\right)};-\mathrm{sgn}\left(\delta\alpha\left(k\right)\right)\theta_{\delta\alpha\left(k\right)},-\mathrm{sgn}\left(\delta\alpha\left(k\right)\right)\eta_{\delta\alpha\left(k\right)}}\right)\textrm{.}
\end{multline*}
For a fixed \(\alpha\), the factor for a given \(k\) vanishes unless \(\iota_{k}=\iota_{\delta\alpha\left(k\right)}\) and \(\mathrm{sgn}\left(k\right)\eta_{k}=-\mathrm{sgn}\left(\delta\alpha\left(k\right)\right)\eta_{\delta\alpha\left(k\right)}\) (equivalent to the conditions on \(\alpha\) in the statement of the theorem) and in addition \(\kappa_{k}=\kappa_{\delta\alpha\left(k\right)}\) and \(\mathrm{sgn}\left(k\right)\theta_{k}=-\mathrm{sgn}\left(\delta\alpha\left(k\right)\right)\theta_{\delta\alpha\left(k\right)}\).  If so, the contribution of pair \(\left\{k,\delta\alpha\left(k\right)\right\}\) is \(\frac{1}{2}\theta_{k}\eta_{k}\), for a total contribution of \(2^{-n}\prod_{k\in\mathrm{FD}\left(\alpha\right)}\theta_{k}\eta_{k}\).

We substitute this expression back into the previous expression.  Considering the sign first, we have a contribution of \(\eta_{k}\theta_{k}\) for each \(k\in\left[n\right]\) and each \(k\in\mathrm{FD}\left(\alpha\right)\).  Multiplying, we have a nontrivial contribution for each \(k\in\pm\left[n\right]\) in exactly one of the two sets, that is, for \(k,-k\) where it is \(-\left|k\right|\in\mathrm{FD}\left(\alpha\right)\).  Reversing the order of summation, we get 
\begin{multline*}
\left(2N\right)^{-n}\sum_{\substack{\alpha\in\mathrm{PM}\left(\pm\left[n\right]\right)\\\iota=\iota\circ\delta\alpha\\\mathrm{sgn}\cdot\eta=\mathrm{sgn}\circ\alpha\cdot\eta\circ\delta\alpha}}\left[\prod_{\substack{k\in\left[n\right]\\-k\in\mathrm{FD}\left(\alpha\right)}}\eta_{k}\eta_{-k}\right]\\\sum_{\substack{\kappa:\pm\left[n\right]\rightarrow\left[M\right]\\\kappa=\kappa\circ\delta\alpha\\\theta:\pm\left[n\right]\rightarrow\left\{1,-1\right\}\\\mathrm{sgn}\cdot\theta=\mathrm{sgn}\circ\alpha\cdot\theta\circ\delta\alpha}}\left[\prod_{\substack{k\in\left[n\right]\\-k\in\mathrm{FD}\left(\alpha\right)}}\theta_{k}\theta_{-k}\right]\prod_{k=1}^{n}D_{\kappa_{k},\kappa_{-k};\theta_{k},\theta_{-k}}\textrm{.}
\end{multline*}
The inner sum is \(\left(2N\right)^{\#\left(\alpha\right)/2}\mathrm{Re}_{\mathrm{FD}\left(\alpha^{-1}\right)}\mathrm{tr}_{\mathrm{FD}\left(\alpha^{-1}\right)}\left(D_{1},\ldots,D_{n}\right)\).  The result follows.
\end{proof}

\section{Haar distribution on the symplectic matrices}
\label{section: Haar}

\subsection{Haar-distributed symplectic matrices}

\begin{definition}
Let \(\mathrm{Sp}\left(N\right)\) be the set of symplectic matrices, i.e.\ the set of matrices \(U\in M_{N\times N}\left(\mathbb{H}\right)\) such that \(U^{\ast}U=I_{N}\).

Such \(U\) form a compact group, so there is a finite Haar measure on \(\mathrm{Sp}\left(N\right)\).  Random matrix \(U:\Omega\rightarrow M_{N\times N}\left(\mathbb{H}\right)\) is a Haar-distributed symplectic matrix if its probability distribution is a Haar measure.
\end{definition}

For more details of the following construction, including a proof of Theorem~\ref{theorem: invariant basis}, see, e.g.\ \cite{MR0000255}, \cite{MR1606831}, Chapter~4.

Let \(V\) be the complex vector space with basis \(\left\{e_{\iota;\eta}:\iota\in\left[N\right],\eta\in\left\{1,-1\right\}\right\}\), and define inner product \(\left(\cdot,\cdot\right)\) by letting \(\left(e_{\iota_{1};\eta_{1}},e_{\iota_{2};\eta_{2}}\right)=\delta_{\iota_{1},\iota_{2}}\delta_{\eta_{1},\eta_{2}}\) and extending conjugate-linearly.  We extend this to in inner product on \(V^{\otimes n}\) by letting \(\left(v_{1}\otimes\cdots\otimes v_{n},w_{1}\otimes\cdots\otimes w_{n}\right)=\left(v_{1},w_{1}\right)\cdots\left(v_{n},w_{n}\right)\).

Let \(S_{n}\) act on \(V^{\otimes n}\) by \(\pi\left(v_{1}\otimes\cdots\otimes v_{n}\right)=v_{\pi\left(1\right)}\otimes\cdots\otimes v_{\pi\left(n\right)}\).  We note that \(\left(\pi\left(\omega_{1}\right),\pi\left(\omega_{2}\right)\right)=\left(\omega_{1},\omega_{2}\right)\).

\begin{theorem}
The subspace of \(V^{\otimes n}\) invariant under \(g^{\otimes n}\) (i.e., the set of \(\omega\in V^{\otimes n}\) such that \(g^{\otimes n}\omega=\omega\)) for all \(g\in\mathrm{Sp}\left(N\right)\) is \(\left\{0\right\}\) if \(n\) is odd, and for \(n\) even is spanned by the images of
\begin{equation}
\sum_{\substack{\iota:\left[n/2\right]\rightarrow\left[N\right]\\\eta:\left[n/2\right]\rightarrow\left\{1,-1\right\}}}\eta_{1}\cdots\eta_{n/2}\left(e_{\iota_{1};\eta_{1}}\otimes e_{\iota_{1};-\eta_{1}}\right)\otimes\cdots\otimes\left(e_{\iota_{n/2};\eta_{n/2}}\otimes e_{\iota_{n/2};-\eta_{n/2}}\right)
\label{formula: first invariant}
\end{equation}
under the action of \(S_{n}\).
\label{theorem: invariant basis}
\end{theorem}

We note that the images of (\ref{formula: first invariant}) under elements of \(B_{n/2}\) are linearly dependent with (\ref{formula: first invariant}), and thus images under elements of the same (left) coset of \(B_{n/2}\) are linearly dependent.  Thus we may choose spanning vectors of the invariant subspace indexed by \({\cal P}_{2}\left(n\right)\), images of (\ref{formula: first invariant}) under a permutation from each coset of \(B_{n/2}\), where cosets may be identified with the image of pairing \(\left\{\left\{1,2\right\},\ldots,\left\{n-1,n\right\}\right\}\) under permutations in that coset (so it is the \(k\)th and \(l\)th tensor factor which are ``entangled'').

We note further \(B_{n/2}\) contains both even and odd permutations, and that an odd permutation of \(B_{n/2}\) multiplies vector (\ref{formula: first invariant}) by \(-1\) while an even permutation does not change its value.  Thus each coset of \(B_{n/2}\) contains even and odd permutations, which map (\ref{formula: first invariant}) to a vector and its additive inverse respectively.

\begin{definition}
For \(\pi\in{\cal P}_{2}\left(n\right)\), let \(e_{\pi}\) be the image of (\ref{formula: first invariant}) under an even permutation such that paired elements of \(\pi\) are entangled.
\end{definition}

\begin{lemma}
For \(\pi_{1},\pi_{2}\in{\cal P}_{2}\left(n\right)\),
\[\left(e_{\pi_{1}},e_{\pi_{2}}\right)=\left(-1\right)^{n/2}\left(-2N\right)^{\#\left(\pi_{1}\vee\pi_{2}\right)}\textrm{.}\]
\label{lemma: inner product}
\end{lemma}
\begin{proof}
For \(\sigma_{1},\sigma_{2}\in S_{n}\) satisfying the hypotheses of Lemma~\ref{lemma: parity}, \(\sigma_{2}\sigma_{1}^{-1}\left(e_{\pi_{1}}\right)=\left(-1\right)^{n/2-\#\left(\pi_{1}\vee\pi_{2}\right)}e_{\pi_{2}}\).  In inner product \(\left(e_{\pi_{1}},\sigma_{2}\sigma_{1}^{-1}\left(e_{\pi_{2}}\right)\right)\), there is a contribution for each \(\iota_{1},\ldots,\iota_{n}\) and \(\eta_{1},\ldots,\eta_{n}\) such that a multiple of \(e_{\iota_{1};\eta_{1}}\otimes\cdots\otimes e_{\iota_{n};\eta_{n}}\) appears in both \(e_{\pi_{1}}\) and \(\sigma_{2}\sigma_{1}^{-1}\left(e_{\pi_{1}}\right)\).  The sign factors \(\eta_{k}\) in \(e_{\pi_{1}}\) and \(\sigma_{2}\sigma_{1}^{-1}\left(e_{\pi_{1}}\right)\) appear with the same indices in both vectors (i.e. the images of the odd numbers), so the sign of any term in their inner product is \(1\).

In such a term, the value of \(\iota\) must be constant on orbits of \(\pi_{i}\), \(i=1,2\), so there are \(N\) possible values of \(\iota\) for each block of \(\pi_{1}\vee\pi_{2}\).  The value of \(\eta\) within a block of \(\pi_{1}\vee\pi_{2}\) is determined by its value on any point: \(\eta_{\pi_{i}\left(k\right)}=-\eta_{k}\), so it is constant on orbits of \(\pi_{1}\pi_{2}\) and has opposite values on the two orbits which form a block of \(\pi_{1}\vee\pi_{2}\); there are thus \(2\) possible values \(\eta\) for each block of \(\pi_{1}\vee\pi_{2}\).  The result follows.
\end{proof}

The following construction of the Weingarten function and the monomial integration formula follows the proofs from \cite{MR2567222}.  See also \cite{MR1959915, MR2217291}.

\begin{lemma}[Collins, Matsumoto]
In a space with inner product \(\left(\cdot,\cdot\right)\), let \(v\) be a vector, \(P\) an orthogonal projection, and \(\left\{v_{1},\ldots,v_{l}\right\}\) a set spanning the image of \(P\).  Define \(\mathrm{Gr}\in M_{n\times n}\left(\mathbb{C}\right)\) by \(\mathrm{Gr}_{ij}=\left(v_{i},v_{j}\right)\), and let \(W\) be a symmetric real matrix satisfying \(\mathrm{Gr}\cdot W\cdot\mathrm{Gr}=I_{n}\) (such as the inverse or pseudoinverse).  Let \(x\in\mathbb{C}^{l}\) with components \(x^{i}=\left(v,v_{i}\right)\) and \(y=Wx\).  Then \(P\left(v\right)=\sum_{i=1}^{l}y^{i}v_{i}\).
\label{lemma: Euclidean exercise}
\end{lemma}

\begin{definition}
If \(\pi\in{\cal P}\left(2n\right)\) has blocks of size \(2\lambda_{1},\ldots,2\lambda_{k}\) with \(\lambda_{1}\geq\cdots\geq\lambda_{k}\), let \(\Lambda\left(\pi\right)\) be the integer partition with parts \(\left(\lambda_{1},\ldots,\lambda_{k}\right)\).  By extension, we can define \(\Lambda\) on a permutation with even cycles (such as the join of pairings or an alternating permutation) by taking its value on the partition of its orbits.
\end{definition}

\begin{definition}
For \(\pi_{+},\pi_{-}\in{\cal P}_{2}\left(n\right)\), we let
\[\mathrm{Gr}\left(\pi_{+},\pi_{-}\right):=\left(e_{\pi_{+}},e_{\pi_{-}}\right)=\left(-1\right)^{n/2}\left(-2N\right)^{\#\left(\pi_{+}\vee\pi_{-}\right)}\textrm{.}\]
The (symplectic) Weingarten function is defined as the inverse (or pseudoinverse) of \(\mathrm{Gr}\), and the entry corresponding to \(\pi_{+},\pi_{-}\) will be denoted \(\mathrm{Wg}\left(\pi_{+},\pi_{-}\right)\) (we will usually suppress \(n\) in the notation, since it will typically be clear from context).

Since the Weingarten function depends only on the sizes of the blocks of \(\pi_{+}\vee\pi_{-}\), we define the Weingarten function of a integer partition \(\lambda\) by \(\mathrm{Wg}\left(\lambda\right):=\mathrm{Wg}\left(\pi_{+},\pi_{-}\right)\) for any \(\pi_{+},\pi_{-}\) with \(\lambda=\Lambda\left(\pi_{+}\vee\pi_{-}\right)\).

We define the normalized Weingarten function
\[\mathrm{wg}\left(\pi_{+},\pi_{-}\right):=\left(-2N\right)^{n-\#\left(\pi_{+}\vee\pi_{-}\right)}\mathrm{Wg}\left(\pi_{+},\pi_{-}\right)\textrm{,}\]
\label{definition: Weingarten function}
and define the normalized Weingarten function on partitions as above.
\end{definition}

\begin{remark}
In this paper, we will always refer to the symplectic Weingarten function unless otherwise stated.  We note that \(\mathrm{Wg}^{\mathrm{Sp}\left(N\right)}=\left(-1\right)^{n/2}\mathrm{Wg}^{O\left(-2N\right)}\) where \(\mathrm{Wg}^{O\left(-2N\right)}\) is the orthogonal Weingarten evaluated at \(-2N\).  See \cite{MR2217291, MR2567222} for tables of values.
\end{remark}

\begin{proposition}[Collins, \'{S}niady]
We may express the Weingarten function:
\begin{multline*}
\mathrm{Wg}\left(\pi_{+},\pi_{-}\right)
\\=\left(2N\right)^{-n/2}\sum_{k\geq 0}\left(-1\right)^{k}\sum_{\substack{\pi_{0},\ldots,\pi_{k}\in{\cal P}_{2}\left(n\right)\\\pi_{0}\neq\pi_{1}\neq\ldots\neq\pi_{k}\\\pi_{0}=\pi_{+},\pi_{k}=\pi_{-}}}\left(-2N\right)^{-\left(d\left(\pi_{0},\pi_{1}\right)+\cdots+d\left(\pi_{k-1},\pi_{k}\right)\right)/2}\textrm{.}
\end{multline*}
\end{proposition}
\begin{proof}
We have:
\begin{multline*}
\mathrm{Wg}=\left(2N\right)^{-n/2}\left(\left(2N\right)^{-n/2}\mathrm{Gr}\right)^{-1}
\\=\left(2N\right)^{-n/2}\left[I_{N}-\left(\left(2N\right)^{-n/2}\mathrm{Gr}-I_{N}\right)+\left(\left(2N\right)^{-n/2}\mathrm{Gr}-I_{N}\right)^{2}-\cdots\right]\textrm{.}
\end{multline*}
The diagonal entries of \(\left(2N\right)^{-n/2}\mathrm{Gr}-I_{N}\) are zero and the off-diagonal entry associated with \(\pi_{+},\pi_{-}\) is \(\left(-2N\right)^{-d\left(\pi_{+},\pi_{-}\right)/2}<1\), so the infinite sum converges.  The result follows.
\end{proof}

\begin{remark}
From \cite{MR2217291}, the normalized Weingarten function
\[\mathrm{wg}\left(\lambda\right)=\prod_{k=1}^{l\left(\lambda\right)}\left(-1\right)^{\lambda_{k}-1}C_{\lambda_{k}-1}+O\left(\frac{1}{N}\right)\]
where \(C_{k}:=\frac{1}{k+1}\binom{2k}{k}\) is the \(k\)th Catalan number.
\label{remark: Catalan}
\end{remark}

\begin{lemma}
The map \(\left(\pi_{+},\pi_{-}\right)\mapsto\pi_{-}\delta\pi_{+}\) is a bijection from \({\cal P}_{2}\left(n\right)^{2}\) to \(\mathrm{PM}\left(n\right)\cap S_{\mathrm{alt}}\left(\pm\left[n\right]\right)\).
\label{lemma: pair of pairings}
\end{lemma}
\begin{proof}
We note that \(\delta\pi_{-}\delta\pi_{+}\) consists of a pairing \(\pi_{+}\) on \(\left[n\right]\) and a pairing \(\delta\pi_{-}\delta\) on \(-\left[n\right]\), so under the bijection in Lemma~\ref{lemma: pairing bijection}, this map is injective.  Since \(\delta\pi_{-}\delta\pi_{+}\) preserves the sign, \(\pi_{-}\delta\pi_{+}\) is alternating, and conversely if \(\alpha\in\mathrm{PM}\left(n\right)\) is alternating, then \(\delta\alpha\) is a pairing preserving sign, i.e.\ of the form \(\delta\pi_{-}\delta\pi_{+}\) for \(\pi_{+},\pi_{-}\in{\cal P}_{2}\left(n\right)\).
\end{proof}

\begin{lemma}[Collins, \'{S}niady]
Let \(U\) be a Haar-distributed symplectic matrix.  Then
\begin{multline*}
\mathbb{E}\left(U_{\iota_{1}\iota_{-1};\eta_{1}\eta_{-1}}\cdots U_{\iota_{n}\iota_{-n};\eta_{n}\eta_{-n}}\right)\\=\sum_{\substack{\alpha\in\mathrm{PM}\left(n\right)\\\iota=\iota\circ\delta\alpha\\\mathrm{sgn}\cdot\eta=\mathrm{sgn}\circ\alpha\cdot\eta\circ\delta\alpha}}\left(2N\right)^{\#\left(\alpha\right)/2-n}\left[\prod_{\substack{k\in\left[n\right]\\-k\in\mathrm{FD}\left(\alpha\right)}}\eta_{k}\eta_{-k}\right]f\left(\alpha\right)
\end{multline*}
where \(f:\mathrm{PM}\left(n\right)\rightarrow\mathbb{C}\) is given by
\[f=\left\{\begin{array}{ll}\mathrm{wg}\left(\Lambda\left(\mathrm{FD}\left(\alpha\right)\right)\right)\textrm{,}&\alpha\in S_{\mathrm{alt}}\left(n\right)\\0\textrm{,}&\textrm{otherwise}\end{array}\right.\textrm{.}\]
\end{lemma}
\begin{proof}
Let
\[P:=\int_{g\in\mathrm{Sp}\left(N\right)}g^{\otimes n}dg:V^{\otimes n}\rightarrow V^{\otimes n}\textrm{.}\]
It is easy to show from the left and right invariance of the probability measure that \(P^{2}=P\) and \(P^{\ast}=P^{-1}=P\); i.e., it is an orthogonal projection.  Its image is spanned by the vectors in Theorem~\ref{theorem: invariant basis} (or is \(\left\{0\right\}\) if \(n\) is odd, in which case the lemma is trivially true): if \(\omega\in V^{\otimes n}\) is invariant under the action of \(g^{\otimes n}\) for \(g\in\mathrm{Sp}\left(N\right)\), then \(P\omega=\omega\); and conversely the action of \(g^{\otimes n}\) on image vector \(P\omega\) can be brought inside the integral, and by left-invariance the integral is again \(P\omega\).

The desired quantity is the entry indexed \(\iota_{1}\iota_{-1}\eta_{1}\eta_{-1},\ldots,\iota_{n}\iota_{-n}\eta_{n}\eta_{-n}\), i.e. \(\left(P\left(e_{\iota_{-1}\eta_{-1}}\otimes\cdots\otimes e_{\iota_{-n}\eta_{-n}}\right),e_{\iota_{1}\eta_{1}}\otimes\cdots\otimes e_{\iota_{n}\eta_{n}}\right)\).  Applying Lemma~\ref{lemma: Euclidean exercise}, we get that
\begin{multline*}
P\left(e_{\iota_{-1}\eta_{-1}}\otimes\cdots\otimes e_{\iota_{-n}\eta_{-n}}\right)\\=\sum_{\left(\pi_{+},\pi_{-}\right)\in{\cal P}_{2}\left(n\right)^{2}}\mathrm{Wg}\left(\pi_{+},\pi_{-}\right)\left(e_{\iota_{-1}\eta_{-1}}\otimes\cdots\otimes e_{\iota_{-n}\eta_{-n}},e_{\pi_{-}}\right)e_{\pi_{+}}
\end{multline*}
and the desired element is
\begin{multline}
\sum_{\left(\pi_{+},\pi_{-}\right)\in{\cal P}_{2}\left(n\right)^{2}}\mathrm{Wg}\left(\pi_{+},\pi_{-}\right)\left(e_{\iota_{-1}\eta_{-1}}\otimes\cdots\otimes e_{\iota_{-n}\eta_{-n}},e_{\pi_{-}}\right)\\\times\left(e_{\pi_{+}},e_{\iota_{1}\eta_{1}}\otimes\cdots\otimes e_{\iota_{n}\eta_{n}}\right)\label{formula: tensor entry}\textrm{.}
\end{multline}

By Lemma~\ref{lemma: pair of pairings}, we can take this sum over \(\alpha\in\mathrm{PM}\left(n\right)\), where the summand vanishes unless \(\alpha\) is alternating.  If \(\alpha\) is alternating, then the term \(e_{\iota_{\pm 1}\eta_{\pm 1}}\otimes\cdots\otimes e_{\iota_{\pm n}\eta_{\pm n}}\) appears in \(e_{\pi_{\pm}}\) only when \(\iota_{\pm k}=\iota_{\pm\pi_{\pm}\left(k\right)}\) and \(\eta_{\pm k}=-\eta_{\pm\pi_{\pm}\left(k\right)}\) for all \(k\in\left[n\right]\), i.e. \(\iota=\iota\circ\delta\pi_{-}\delta\pi_{+}\) and \(\eta=-\eta\circ\delta\pi_{-}\delta\pi_{+}\); otherwise the inner product vanishes, giving the conditions on \(\alpha\).

By Lemma~\ref{lemma: parity}, it is possible to find permutations \(\sigma_{+}\) and \(\sigma_{-}\) such that \(\sigma_{\pm}\left(\left\{\left\{1,2\right\},\cdots\left\{n-1,n\right\}\right\}\right)=\pi_{\pm}\) mapping the odd integers to the \(-k\in\mathrm{FD}\left(\alpha\right)\) with \(k>0\).  The sign \(\prod_{k>0:-k\in\mathrm{FD}\left(\alpha\right)}\eta_{k}\eta_{-k}\) then differs from the sign of the product of the two inner products in (\ref{formula: tensor entry}) by \(\mathrm{sgn}\left(\sigma_{1}\sigma_{2}\right)=\left(-1\right)^{n/2-\#\left(\alpha\right)/2}\) (since the \(e_{\pm\pi}\) are produced by even permutations).  The result follows.
\end{proof}

\subsection{Symplectically Invariant Matrices}
\label{subsection: symplectically invariant}

\begin{definition}
Random matrices \(X_{1},\ldots,X_{n}:\Omega\rightarrow M_{N\times N}\left(\mathbb{H}\right)^{n}\) are {\em symplectically invariant} if, for any \(U\in\mathrm{Sp}\left(n\right)\), the joint probability distribution of \(U^{\ast}X_{1}U,\cdots,U^{\ast}X_{N}U\) is the same as that of \(X_{1},\ldots,X_{n}\).
\end{definition}

\begin{definition}[Capitaine, Casalis]
Let \(X_{1},\ldots,X_{n}:\Omega\rightarrow M_{N\times N}\left(\mathbb{H}\right)\) be random quaternionic matrices.  The matrix cumulant corresponding to \(\alpha\in\mathrm{PM}\left(n\right)\) is
\[\sum_{\pi\in\mathrm{PM}\left(n\right)}\left(-1\right)^{\chi\left(\alpha,\pi\right)}2^{\#\left(\pi\right)/2}\mathrm{Wg}\left(\mathrm{FD}\left(\alpha\pi^{-1}\right)\right)\mathbb{E}\left[\mathrm{Re}_{\mathrm{FD}\left(\pi\right)}\mathrm{Tr}_{\mathrm{FD}\left(\pi\right)}\left(X_{1},\ldots,X_{n}\right)\right]\textrm{.}\]
See in particular \cite{MR2483727}, also \cite{MR2240781, MR2337139}.

We define the {\em normalized matrix cumulant} to be \(\left(2N\right)^{n-\#\left(\alpha\right)/2}\) times the matrix cumulant: 
\[\sum_{\pi\in\mathrm{PM}\left(n\right)}\left(-2N\right)^{\chi\left(\alpha,\pi\right)-\#\left(\alpha\right)}\mathrm{wg}\left(\mathrm{FD}\left(\alpha\pi^{-1}\right)\right)\mathbb{E}\left[\mathrm{Re}_{\mathrm{FD}\left(\pi\right)}\mathrm{tr}_{\mathrm{FD}\left(\pi\right)}\left(X_{1},\ldots,X_{n}\right)\right]\textrm{.}\]
\end{definition}

\begin{remark}
The cumulants defined here are equivalent to those defined in \cite{MR2483727}; however, we index over \(\mathrm{PM}\left(n\right)\) rather than \({\cal P}_{2}\left(\pm\left[n\right]\right)\) (see Lemma~\ref{lemma: pairing bijection}).  As in \cite{MR3217665}, we use a slightly different convolution than \cite{MR2483727}, but the cumulants themselves are equal.
\end{remark}

\begin{proposition}
Let random matrices \(X_{1},\ldots,X_{n}\) be symplectically invariant.  Then
\begin{multline*}
\mathbb{E}\left(X^{\left(1\right)}_{\iota_{1},\iota_{-1};\eta_{1},\eta_{-1}}\cdots X^{\left(n\right)}_{\iota_{n},\iota_{-n};\eta_{n}\eta_{-n}}\right)
\\=\sum_{\substack{\alpha\in\mathrm{PM}\left(n\right)\\\iota=\iota\circ\delta\alpha\\\mathrm{sgn}\cdot\eta=\mathrm{sgn}\circ\alpha\cdot\eta\circ\delta\alpha}}\left(2N\right)^{\#\left(\alpha\right)/2-n}\left[\prod_{\substack{k\in\left[n\right]\\-k\in\mathrm{FD}\left(\alpha\right)}}\eta_{k}\eta_{-k}\right]f\left(\alpha\right)
\end{multline*}
where \(f:\mathrm{PM}\left(n\right)\rightarrow\mathbb{C}\) is the normalized matrix cumulant.
\label{proposition: symplectically invariant}
\end{proposition}
\begin{proof}
Let \(U:\Omega\rightarrow M_{N\times N}\left(\mathbb{H}\right)\) be a random Haar-distributed symplectic matrix independent from the \(X_{k}\).  Since the distribution of \(X_{1},\ldots,X_{n}\) is the same as that of \(U^{\ast}X_{1}U,\ldots,U^{\ast}X_{n}U\), the desired expected value is equal to the expected value of
\begin{multline*}
\left(U^{\ast}X_{1}U\right)_{\iota_{1}\iota_{-1};\eta_{1}\eta_{-1}}\cdots\left(U^{\ast}X_{n}U\right)_{\iota_{n}\iota_{-n};\eta_{n}\eta_{-n}}
\\=\sum_{\substack{\kappa:\pm\left[n\right]\rightarrow\left[N\right]\\\theta:\pm\left[n\right]\rightarrow\left\{1,-1\right\}}}\prod_{k=1}^{n}\eta_{k}\theta_{k}U_{\kappa_{k},\iota_{k};-\theta_{k},-\eta_{k}}U_{\kappa_{-k},\iota_{-k};\theta_{-k},\eta_{-k}}X^{\left(k\right)}_{\kappa_{k}\kappa_{-k};\theta_{k}\theta_{-k}}\textrm{.}
\end{multline*}

Fix \(\kappa\) and \(\theta\).  The expected value of the product of the \(2n\) entries from \(U\) is a sum over \({\cal P}_{2}\left(\pm\left[n\right]\right)^{2}\), which is in bijection with \(\mathrm{PM}\left(n\right)^{2}\):
\begin{multline*}
\mathbb{E}\left[\prod_{k\in\pm\left[n\right]}U_{\kappa_{k},\iota_{k};-\mathrm{sgn}\left(k\right)\theta_{k},-\mathrm{sgn}\left(k\right)\eta_{k}}\right]
\\=\sum_{\substack{\alpha\in\mathrm{PM}\left(n\right)\\\iota=\iota\circ\delta\alpha\\\mathrm{sgn}\cdot\eta=\mathrm{sgn}\circ\alpha\cdot\eta\circ\delta\alpha}}\sum_{\substack{\pi\in\mathrm{PM}\left(n\right)\\\kappa=\kappa\circ\delta\pi\\\mathrm{sgn}\cdot\theta=\mathrm{sgn}\circ\pi\cdot\theta\circ\delta\theta}}\left[\prod_{k\in I}\theta_{k}\eta_{k}\right]\left(2N\right)^{\#\left(\alpha^{-1}\pi\right)/2-2n}\mathrm{wg}\left(\delta\alpha,\delta\pi\right)\textrm{.}
\end{multline*}
where \(I\subseteq\pm\left[n\right]\) is a set with exactly one element in each pair of \(\delta\alpha\) and each pair of \(\delta\pi\).  We note that such a set \(I\) is a choice of one cycle from each pair of cycles in \(\delta\alpha\cdot\delta\pi=\alpha^{-1}\pi\) described in Lemma~\ref{lemma: pairings}.  The other choice (which contains the other element of each pair, for both pairings \(\delta\alpha\) and \(\delta\pi\)) gives the same value, since (by the constraints on the indices \(\eta_{k}\) and \(\theta_{k}\)) the partnered index has the same value when the signs of the \(k\) are different and opposite value when the signs of the \(k\) are the same, and there are an even number of such pairs (since the loop is formed from an even number of pairs, and there must be an even number of pairs which pair a positive integer with a negative one).  Thus the choice of which \(I\) containing one element of each pair is arbitrary.

We would like the sign
\[\prod_{\substack{k\in\left[n\right]\\-k\in\mathrm{FD}\left(\alpha\right)}}\eta_{k}\eta_{-k}\prod_{\substack{k\in\left[n\right]\\-k\in\mathrm{FD}\left(\pi\right)}}\theta_{k}\theta_{-k}\]
to appear (to give the sign in the statement and the sign required to express the elements of \(X\) as a trace), which differs from the sign \(\prod_{k=1}^{n}\eta_{k}\theta_{k}\) by \(\prod_{k\in\mathrm{FD}\left(\alpha\right)}\eta_{k}\prod_{k\in\mathrm{FD}\left(\pi\right)}\theta_{k}\).  This sign differs from \(\prod_{k\in I}\eta_{k}\theta_{k}\) by the number of pairs in \(\delta\alpha\) or \(\delta\pi\) whose elements have the same sign, and whose distinguished element in \(I\) is different than the distinguished element in \(\mathrm{FD}\left(\alpha\right)\) or \(\mathrm{FD}\left(\pi\right)\).  The calculation is similar to the one in the proof of Proposition~\ref{proposition: topological expansion}.  The total number of pairs whose distinguished element is different is the product of the signs of permutations mapping \(\delta\alpha\mapsto\delta\alpha\) and \(\delta\pi\mapsto\delta\pi\) taking \(I\) to \(\mathrm{FD}\left(\alpha\right)\) and \(\mathrm{FD}\left(\pi\right)\) respectively.  This is the same as the product of the signs of the permutations mapping \(\delta\alpha\mapsto\delta\pi\) taking \(I\) to \(I\) and \(\mathrm{FD}\left(\alpha\right)\) to \(\mathrm{FD}\left(\pi\right)\) respectively, which we calculate (by Corollary~\ref{corollary: parity}) to be \(\left(-1\right)^{\chi\left(\alpha,\pi\right)+\left|\left[n\right]\cap\mathrm{FD}\left(\alpha\right)\right|+\left|\left[n\right]\cap\mathrm{FD}\left(\pi\right)\right|}\) (taking the latter permutation via \(\delta\) with positive integers distinguished).  The parity of the number of pairs in \(\delta\alpha\) (resp.\ \(\delta\pi\)) which contain both a positive and a negative integer and whose distinguished elements are different is the parity of \(\left|\left[n\right]\cap\mathrm{FD}\left(\alpha\right)\right|-\left|\left[n\right]\cap I\right|\) (resp.\ \(\left|\left[n\right]\cap\mathrm{FD}\left(\pi\right)\right|-\left|\left[n\right]\cap I\right|\), giving a final sign of \(\left(-1\right)^{\chi\left(\alpha,\pi\right)}\).

Moving the sum over \(\kappa\) and \(\theta\) inside the other sums, we get a summand for \(\alpha\in\mathrm{PM}\left(n\right)\) with \(\iota=\iota\circ\delta\alpha\) and \(\mathrm{sgn}\cdot\eta=\mathrm{sgn}\circ\alpha\cdot\eta\circ\delta\alpha\):
\begin{multline*}
\left[\prod_{\substack{k\in\left[n\right]\\-k\in\mathrm{FD}\left(\alpha\right)}}\eta_{k}\eta_{-k}\right]\sum_{\pi\in\mathrm{PM}\left(n\right)}\left(-1\right)^{\chi\left(\alpha,\pi\right)}\left(2N\right)^{\#\left(\alpha^{-1}\pi\right)/2+\#\left(\pi\right)/2-2n}\mathrm{wg}\left(\delta\alpha,\delta\pi\right)
\\\times\mathbb{E}\left[\mathrm{Re}_{\pi}\mathrm{tr}_{\pi}\left(X_{1},\ldots,X_{n}\right)\right]
\end{multline*}
from which the result follows.
\end{proof}

\begin{remark}
By Remark~\ref{remark: Catalan}, if \(\lim_{N\rightarrow\infty}\mathbb{E}\left[\mathrm{Re}_{\pi}\mathrm{tr}_{\pi}\left(X_{1},\ldots,X_{n}\right)\right]\) exists (as it does for all of the ensembles discussed in this paper), the highest order terms will be those for which the Euler characteristic \(\chi\left(\alpha,\pi\right)\) is large, that is those where \(\pi\) does not connect cycles of \(\alpha\) and is planar with respect to \(\alpha\) (Definition~\ref{definition: planar}).
\end{remark}

\subsection{Matrices Symplectically in General Position}

\begin{definition}
We say that random matrices \(X_{1},\ldots,X_{n}\) and \(Y_{1},\ldots,Y_{n}\) are {\em symplectically in general position} if the joint probability distribution is the same when one of the sets is conjugated by an arbitrary symplectic matrix.

We will be considering matrices which are symplectically invariant and symplectically in general position, so the distribution is the same when each ensemble is conjugated by a different arbitrary symplectic matrix.
\end{definition}

\begin{proposition}
Let \(w:\left[n\right]\rightarrow\left[C\right]\) be a word in colours \(\left[C\right]\), and let the \(X_{k}\), \(w\left(k\right)=c\) be symplectically invariant and symplectically in general position for each \(c\in\left[C\right]\).  Then
\begin{multline*}
\mathbb{E}\left(X^{\left(1\right)}_{\iota_{1}\iota_{-1};\eta_{1}\eta_{-1}}\cdots X^{\left(n\right)}_{\iota_{n}\iota_{-n};\eta_{n}\eta_{-n}}\right)
\\=\sum_{\substack{\alpha\in\mathrm{PM}\left(n\right)\\\iota=\iota\circ\delta\alpha\\\mathrm{sgn}\cdot\eta=\mathrm{sgn}\circ\alpha\cdot\eta\circ\delta\alpha}}\left(2N\right)^{\#\left(\alpha\right)/2-n}\left[\prod_{\substack{k\in\left[n\right]\\-k\in\mathrm{FD}\left(\alpha\right)}}\eta_{k}\eta_{-k}\right]f\left(\alpha\right)
\end{multline*}
where \(f\left(\alpha\right)\rightarrow\mathbb{C}\) vanishes unless \(\alpha=\alpha_{1},\ldots,\alpha_{C}\), \(\alpha_{c}\in\mathrm{PM}\left(w^{-1}\left(c\right)\right)\), in which case
\begin{multline*}
f\left(\alpha\right)=\sum_{\substack{\pi=\pi_{1}\cdots\pi_{C}\\\pi_{c}\in\mathrm{PM}\left(w^{-1}\left(c\right)\right)}}\left(-2N\right)^{\chi\left(\alpha,\pi\right)-\#\left(\alpha\right)}\mathrm{wg}\left(\delta\alpha_{1},\delta\pi_{1}\right)\cdots\mathrm{wg}\left(\delta\alpha_{C},\delta\pi_{C}\right)\\\times\mathbb{E}\left(\mathrm{Re}_{\pi}\mathrm{tr}_{\pi}\left(X_{1},\ldots,X_{n}\right)\right)\textrm{.}
\end{multline*}
\end{proposition}
\begin{proof}
The proof is similar to that of Proposition~\ref{proposition: symplectically invariant}.  Let \(U_{1},\ldots,U_{C}:\Omega\rightarrow M_{N\times N}\left(\mathbb{H}\right)\) be independent Haar distributed symplectic matrices.  Because the matrices are symplectically in general position, the expected value is equal to the expected value of
\begin{multline*}
\left(U_{w\left(1\right)}^{\ast}X_{1}U_{w\left(1\right)}\right)_{\iota_{1}\iota_{-1};\eta_{1}\eta_{-1}}\cdots\left(U_{w\left(n\right)}^{\ast}X_{n}U_{w\left(n\right)}\right)_{\iota_{n}\iota_{-n};\eta_{n}\eta_{-n}}
\\=\sum_{\substack{\kappa:\pm\left[n\right]\rightarrow\left[N\right]\\\theta:\pm\left[n\right]\rightarrow\left\{1,-1\right\}}}\prod_{k=1}^{n}\eta_{k}\theta_{k}U^{\left(w\left(k\right)\right)}_{\kappa_{k},\iota_{k};-\theta_{k},-\eta_{k}}U^{\left(w\left(k\right)\right)}_{\kappa_{-k},\iota_{-k};\theta_{-k}\eta_{-k}}X^{\left(k\right)}_{\kappa_{k}\kappa_{-k};\theta_{k}\theta_{-k}}\textrm{.}
\end{multline*}
For a fixed \(\kappa\) and \(\theta\), we can express the expected value of the entries of the \(U_{k}\) as a sum over \(\prod_{c\in\left[C\right]}{\cal P}_{2}\left(\pm w^{-1}\left(c\right)\right)^{2}\), which is in bijection with \(\prod_{c\in\left[C\right]}\mathrm{PM}\left(w^{-1}\left(c\right)\right)^{2}\).  The expected value of the entries from the \(U_{c}\) is:
\begin{multline*}
\mathbb{E}\left(\prod_{k\in\pm\left[n\right]}U^{\left(\mathrm{sgn}\left(k\right)w\left(\left|k\right|\right)\right)}_{\kappa_{k}\iota{k};-\mathrm{sgn}\left(k\right)\theta_{k}-\mathrm{sgn}\left(k\right)\eta_{-k}}\right)
\\=\sum_{\substack{\alpha=\alpha_{1}\cdots\alpha_{C}\\\alpha_{c}\in\mathrm{PM}\left(w^{-1}\left(c\right)\right)\\\iota=\iota\circ\delta\alpha\\\mathrm{sgn}\cdot\eta=\mathrm{sgn}\circ\alpha\cdot\eta\circ\delta\alpha}}\sum_{\substack{\pi=\pi_{1}\cdots\pi_{C}\\\pi_{c}\in\mathrm{PM}\left(w^{-1}\left(c\right)\right)\\\kappa=\kappa\circ\delta\pi\\\mathrm{sgn}\cdot\theta=\mathrm{sgn}\circ\pi\cdot\theta\circ\delta\pi}}\left[\prod_{k\in I}\theta_{k}\eta_{k}\right]\left(2N\right)^{\#\left(K\left(\alpha,\pi\right)/2\right)-2n}
\\\times\mathrm{wg}\left(\delta\alpha_{1},\delta\pi_{1}\right)\cdots\mathrm{wg}\left(\delta\alpha_{C},\delta\pi_{C}\right)
\end{multline*}
where \(I\subseteq\pm\left[n\right]\) contains exactly one element from each pair in \(\delta\alpha\) and each pair in \(\delta\pi\).  The calculation of the sign is as in the proof of Proposition~\ref{proposition: symplectically invariant}, as is the constraints on the indices.  The result follows.
\end{proof}

\section{Acknowledgements}

I would like to thank Dr.\ W\l odzimierz Bryc for some helpful notes on the quaternionic Wick formula.  I would also like to that Dr.\ Roe Goodman and Dr.\ Nolan R.\ Wallach for explanation of the invariant spaces of the symplectic matrices.  Finally, I would like to thank D.\ P.\ Leaman for helping to hunt down references for me.

\bibliography{paper}
\bibliographystyle{plain}

\end{document}

%% file: real.pdf_t
\begin{picture}(0,0)%
\includegraphics{real.pdf}%
\end{picture}%
\setlength{\unitlength}{3947sp}%
\begingroup\makeatletter\ifx\SetFigFont\undefined%
\gdef\SetFigFont#1#2#3#4#5{%
  \reset@font\fontsize{#1}{#2pt}%
  \fontfamily{#3}\fontseries{#4}\fontshape{#5}%
  \selectfont}%
\fi\endgroup%
\begin{picture}(6267,3816)(1232,-3765)
\put(5531,-2426){\makebox(0,0)[b]{\smash{{\SetFigFont{10}{12.0}{\familydefault}{\mddefault}{\updefault}{\color[rgb]{0,0,0}\(Y_{8}\)}%
}}}}
\put(5226,-2076){\makebox(0,0)[b]{\smash{{\SetFigFont{10}{12.0}{\familydefault}{\mddefault}{\updefault}{\color[rgb]{0,0,0}\(U\)}%
}}}}
\put(5996,-2516){\makebox(0,0)[b]{\smash{{\SetFigFont{10}{12.0}{\familydefault}{\mddefault}{\updefault}{\color[rgb]{0,0,0}\(Z_{2}\)}%
}}}}
\put(3261,-1621){\makebox(0,0)[b]{\smash{{\SetFigFont{10}{12.0}{\familydefault}{\mddefault}{\updefault}{\color[rgb]{0,0,0}\(Y_{4}\)}%
}}}}
\put(2086,-2366){\makebox(0,0)[b]{\smash{{\SetFigFont{10}{12.0}{\familydefault}{\mddefault}{\updefault}{\color[rgb]{0,0,0}\(U\)}%
}}}}
\put(2711,-2501){\makebox(0,0)[b]{\smash{{\SetFigFont{10}{12.0}{\familydefault}{\mddefault}{\updefault}{\color[rgb]{0,0,0}\(Y_{3}\)}%
}}}}
\put(3016,-2096){\makebox(0,0)[b]{\smash{{\SetFigFont{10}{12.0}{\familydefault}{\mddefault}{\updefault}{\color[rgb]{0,0,0}\(W\)}%
}}}}
\put(5131,-1616){\makebox(0,0)[b]{\smash{{\SetFigFont{10}{12.0}{\familydefault}{\mddefault}{\updefault}{\color[rgb]{0,0,0}\(Y_{7}\)}%
}}}}
\put(5276,-1166){\makebox(0,0)[b]{\smash{{\SetFigFont{10}{12.0}{\familydefault}{\mddefault}{\updefault}{\color[rgb]{0,0,0}\(Z_{1}\)}%
}}}}
\put(5566,-791){\makebox(0,0)[b]{\smash{{\SetFigFont{10}{12.0}{\familydefault}{\mddefault}{\updefault}{\color[rgb]{0,0,0}\(Y_{6}\)}%
}}}}
\put(6016,-706){\makebox(0,0)[b]{\smash{{\SetFigFont{10}{12.0}{\familydefault}{\mddefault}{\updefault}{\color[rgb]{0,0,0}\(U\)}%
}}}}
\put(6521,-836){\makebox(0,0)[b]{\smash{{\SetFigFont{10}{12.0}{\familydefault}{\mddefault}{\updefault}{\color[rgb]{0,0,0}\(Y_{5}\)}%
}}}}
\put(6746,-1156){\makebox(0,0)[b]{\smash{{\SetFigFont{10}{12.0}{\familydefault}{\mddefault}{\updefault}{\color[rgb]{0,0,0}\(I\)}%
}}}}
\put(6921,-1626){\makebox(0,0)[b]{\smash{{\SetFigFont{10}{12.0}{\familydefault}{\mddefault}{\updefault}{\color[rgb]{0,0,0}\(Y_{10}\)}%
}}}}
\put(6451,-2451){\makebox(0,0)[b]{\smash{{\SetFigFont{10}{12.0}{\familydefault}{\mddefault}{\updefault}{\color[rgb]{0,0,0}\(Y_{9}\)}%
}}}}
\put(6686,-2086){\makebox(0,0)[b]{\smash{{\SetFigFont{10}{12.0}{\familydefault}{\mddefault}{\updefault}{\color[rgb]{0,0,0}\(Z_{2}^{\ast}\)}%
}}}}
\put(2211,-826){\makebox(0,0)[b]{\smash{{\SetFigFont{10}{12.0}{\familydefault}{\mddefault}{\updefault}{\color[rgb]{0,0,0}\(U^{\ast}\)}%
}}}}
\put(1871,-1606){\makebox(0,0)[b]{\smash{{\SetFigFont{10}{12.0}{\familydefault}{\mddefault}{\updefault}{\color[rgb]{0,0,0}\(Y_{2}\)}%
}}}}
\put(1596,-1606){\makebox(0,0)[b]{\smash{{\SetFigFont{10}{12.0}{\familydefault}{\mddefault}{\updefault}{\color[rgb]{0,0,0}\(Z_{1}\)}%
}}}}
\put(2581,-1611){\makebox(0,0)[b]{\smash{{\SetFigFont{10}{12.0}{\familydefault}{\mddefault}{\updefault}{\color[rgb]{0,0,0}\(Y_{1}\)}%
}}}}
\end{picture}%

%% file: trace.pdf_t
\begin{picture}(0,0)%
\includegraphics{trace.pdf}%
\end{picture}%
\setlength{\unitlength}{3947sp}%
\begingroup\makeatletter\ifx\SetFigFont\undefined%
\gdef\SetFigFont#1#2#3#4#5{%
  \reset@font\fontsize{#1}{#2pt}%
  \fontfamily{#3}\fontseries{#4}\fontshape{#5}%
  \selectfont}%
\fi\endgroup%
\begin{picture}(6305,3868)(1194,-3765)
\put(5226,-2076){\makebox(0,0)[b]{\smash{{\SetFigFont{10}{12.0}{\familydefault}{\mddefault}{\updefault}{\color[rgb]{0,0,0}\(U\)}%
}}}}
\put(5996,-2516){\makebox(0,0)[b]{\smash{{\SetFigFont{10}{12.0}{\familydefault}{\mddefault}{\updefault}{\color[rgb]{0,0,0}\(Z_{2}\)}%
}}}}
\put(3261,-1621){\makebox(0,0)[b]{\smash{{\SetFigFont{10}{12.0}{\familydefault}{\mddefault}{\updefault}{\color[rgb]{0,0,0}\(Y_{4}\)}%
}}}}
\put(1661,-1061){\makebox(0,0)[b]{\smash{{\SetFigFont{10}{12.0}{\familydefault}{\mddefault}{\updefault}{\color[rgb]{0,0,0}\(Y_{1}\)}%
}}}}
\put(1666,-1586){\makebox(0,0)[b]{\smash{{\SetFigFont{10}{12.0}{\familydefault}{\mddefault}{\updefault}{\color[rgb]{0,0,0}\(Z_{1}\)}%
}}}}
\put(1671,-2156){\makebox(0,0)[b]{\smash{{\SetFigFont{10}{12.0}{\familydefault}{\mddefault}{\updefault}{\color[rgb]{0,0,0}\(Y_{2}\)}%
}}}}
\put(2086,-2366){\makebox(0,0)[b]{\smash{{\SetFigFont{10}{12.0}{\familydefault}{\mddefault}{\updefault}{\color[rgb]{0,0,0}\(U\)}%
}}}}
\put(2711,-2501){\makebox(0,0)[b]{\smash{{\SetFigFont{10}{12.0}{\familydefault}{\mddefault}{\updefault}{\color[rgb]{0,0,0}\(Y_{3}\)}%
}}}}
\put(3016,-2096){\makebox(0,0)[b]{\smash{{\SetFigFont{10}{12.0}{\familydefault}{\mddefault}{\updefault}{\color[rgb]{0,0,0}\(W\)}%
}}}}
\put(5131,-1616){\makebox(0,0)[b]{\smash{{\SetFigFont{10}{12.0}{\familydefault}{\mddefault}{\updefault}{\color[rgb]{0,0,0}\(Y_{7}\)}%
}}}}
\put(5276,-1166){\makebox(0,0)[b]{\smash{{\SetFigFont{10}{12.0}{\familydefault}{\mddefault}{\updefault}{\color[rgb]{0,0,0}\(Z_{1}\)}%
}}}}
\put(5566,-791){\makebox(0,0)[b]{\smash{{\SetFigFont{10}{12.0}{\familydefault}{\mddefault}{\updefault}{\color[rgb]{0,0,0}\(Y_{6}\)}%
}}}}
\put(6016,-706){\makebox(0,0)[b]{\smash{{\SetFigFont{10}{12.0}{\familydefault}{\mddefault}{\updefault}{\color[rgb]{0,0,0}\(U\)}%
}}}}
\put(6521,-836){\makebox(0,0)[b]{\smash{{\SetFigFont{10}{12.0}{\familydefault}{\mddefault}{\updefault}{\color[rgb]{0,0,0}\(Y_{5}\)}%
}}}}
\put(6746,-1156){\makebox(0,0)[b]{\smash{{\SetFigFont{10}{12.0}{\familydefault}{\mddefault}{\updefault}{\color[rgb]{0,0,0}\(I\)}%
}}}}
\put(6451,-2451){\makebox(0,0)[b]{\smash{{\SetFigFont{10}{12.0}{\familydefault}{\mddefault}{\updefault}{\color[rgb]{0,0,0}\(Y_{9}\)}%
}}}}
\put(6686,-2086){\makebox(0,0)[b]{\smash{{\SetFigFont{10}{12.0}{\familydefault}{\mddefault}{\updefault}{\color[rgb]{0,0,0}\(Z_{2}^{\ast}\)}%
}}}}
\put(2211,-826){\makebox(0,0)[b]{\smash{{\SetFigFont{10}{12.0}{\familydefault}{\mddefault}{\updefault}{\color[rgb]{0,0,0}\(U^{\ast}\)}%
}}}}
\put(5941,-1486){\makebox(0,0)[b]{\smash{{\SetFigFont{10}{12.0}{\familydefault}{\mddefault}{\updefault}{\color[rgb]{0,0,0}\(Y_{8}\)}%
}}}}
\put(6291,-2016){\makebox(0,0)[b]{\smash{{\SetFigFont{10}{12.0}{\familydefault}{\mddefault}{\updefault}{\color[rgb]{0,0,0}\(Y_{10}\)}%
}}}}
\end{picture}%